\documentclass{amsart}
\usepackage[dvips]{epsfig}
\usepackage{graphicx}
\usepackage{latexsym}
\usepackage{amsmath}
\usepackage{amsthm}
\usepackage{amssymb}
\RequirePackage[colorlinks=true,citecolor=blue,urlcolor=blue]{hyperref}
\usepackage{scrextend}
\usepackage{hyperref}
\usepackage{multirow}
\usepackage{booktabs}
\usepackage{xcolor}
\usepackage{ulem}

\usepackage{caption}
\usepackage{subcaption}
\usepackage{float}

 \usepackage[applemac]{inputenc}

\newtheorem{theorem}{Theorem}
\newtheorem{assumption}[theorem]{Assumption}
\newtheorem{lemma}[theorem]{Lemma}

\newtheorem{proposition}[theorem]{Proposition}
\newtheorem{remark}[theorem]{Remark}

\newcommand{\vip}{\vskip.2cm}
\newcommand{\R}{{\mathbb{R}}}

\newcommand{\E}{\mathbb{E}}
\newcommand{\COMMENTAIRE}[1]{}
\newcommand{\PP}{{\mathbb{P}}}

\newcommand\numberthis{\addtocounter{equation}{1}\tag{\theequation}}

\begin{document}

\title[Estimating fast mean-reverting jumps]{Estimating fast mean-reverting jumps in electricity market models}

\author{Thomas Deschatre, Olivier F\'eron and Marc Hoffmann}

\address{Thomas Deschatre, EDF Lab Paris-Saclay and Universit\'e Paris-Dauphine PSL, CNRS, Ceremade, 75016 Paris, France}

\email{thomas-t.deschatre@edf.fr}

\address{Olivier F\'eron, EDF Lab Paris-Saclay and FiME, Laboratoire de Finance des March\'es de l'Energie, 91120 Palaiseau, France.}

\email{olivier-2.feron@edf.fr}

\address{Marc Hoffmann, Universit\'e Paris-Dauphine PSL, CNRS, Ceremade, 75016 Paris, France. (to whom correspondence should be addressed.)}

\email{hoffmann@ceremade.dauphine.fr}

\begin{abstract}
Based on empirical evidence of fast mean-reverting spikes, electricity spot prices are often modeled $X+Z^\beta$ as the sum of a continuous It\^o semimartingale $X$ and a  mean-reverting compound Poisson process $Z_t^\beta = \int_0^t\int_{\R} xe^{-\beta(t-s)}\underline{p}(ds,dt)$ where $\underline{p}(ds,dt)$ is Poisson random measure with intensity $\lambda ds\otimes dt$. In a first part, we investigate the estimation of $(\lambda,\beta)$  from discrete observations and establish asymptotic efficiency in various asymptotic settings. In a second part, we discuss the use of our inference results for correcting the value of forward contracts on electricity markets in presence of spikes. We implement our method on real data in the French, German and Australian market over 2015 and 2016 and show in particular the effect of spike modelling on the valuation of certain strip options. In particular, we show that some out-of-the-money options have a significant value if we incorporate spikes in our modelling, while having a value close to $0$ otherwise. 
\end{abstract}

\maketitle




\section{Introduction}

\subsection{Motivation} \label{subsec:motiv}
A striking empirical feature of electricity spot prices is the presence of spikes, that can be described by a jump in the price process immediately followed by a fast mean reversion (see Figure \ref{spotdata} showing the behaviour of electricity spot prices in different markets over one year of historical data). These spikes are due to the non-storability of electricity, an abrupt change in the demand or the offer (due to weather conditions, outages and so on) having a direct impact on prices. For risk management purposes, the modelling of these extreme events is essential. And, due to the non-storability of electricity, the modelling of forward contracts (used as hedging products) are also needed.  
If $(S_t)_{t \geq 0}$ denotes the electricity spot price, the forward price $f(t,T)$ at time $t$ delivering 1 megawatt hour (MWh) at time $T$ can be defined as
\begin{equation} \label{value forward contracts}
f(t,T) = \mathbb{E}\big[S_T\,\big|\,\mathcal{F}_t\big],\;\;t\geq 0
\end{equation}
where $\mathcal{F}_t$ is the available information up to time $t$ and the expectation is taken under a risk-neutral probability. In this context, one usually faces two major issues: 
	first and prior to data analysis, a stochastic model that captures the main characteristics of spot prices, and especially the presence of fast mean-reverting spikes has to be set, however simple enough to give tractable formulas for the forward prices $f(t,T)$.
	Second, the chosen model must be calibrated with efficient statistical procedures to show its adequacy to the data, and to properly quantify risk measures. The main difficulty is the estimation of the characteristics of the spikes.

\medskip
Several models for spikes in electricity spot prices have been proposed in the litterature. Cartea and Figueroa \cite{cartea05} extend the popular and tractable approach of Lucia {\it et al.} \cite{lucia02} by introducing jumps in the price process, resulting in the model
$$
\log S_t = \rho(t) + Y_t,\;\;dY_t = -\beta Y_tdt+\sigma(t)dW_t+\log J dN_t,\;\;t\geq 0, 
$$
where $\rho(t)$ and $\sigma(t)$ are deterministic components, $(W_t)_{t \geq 0}$ is a Wiener process, $(N_t)_{t \geq 0}$ is a Poisson process and $J$ is the jump size drawn proportional to a log-normal distribution.     
A similar model is proposed in Geman and Roncoroni \cite{geman06}, adding up a threshold parameter that determines the sign of the jumps. In these approaches, the mean reverting coefficient $\beta >0$ is the same for the continuous component and for the spike component. However, statistical evidence shows that the mean reversion of the spike component is much stronger than the one of the Brownian component, see for instance Benth {\it et al.}  \cite{benth12}. The estimated $\beta$ then underestimates the mean reversion of the spike component and overestimates the one of the continuous component. A similar model slightly more realistic is also proposed by Geman and Roncoroni \cite{geman06} but this one does not provide explicit formulas for deriving $f(t,T)$. 
Yet another approach is undertaken in Benth {\it et al.} \cite{benth07,benth12} with multi-factor models:
$$
S_t = \sum_{i=1}^m w_i Y^i_t,\;\;dY^i_t = -\beta_i Y^i_t dt + dL^i_t,\;\;t\geq 0,\;\;i=1, \ldots, m,
$$
for some weights $w_i$, and where $(L^i_t)_{t \geq 0}$ are independent time-inhomogeneous subordinators ensuring that $(S_t)_{t \geq 0}$ remains nonnegative. Benth and collaborators establish in \cite{benth07,benth12} that $m=2$ is sufficient for modelling purposes, each factor $(Y_t^i)_{t \geq 0}$ having its own mean reverting parameter, allowing for a fast mean reversion and a slower one. However, the use of subordinators implies that the volatility of the process seems to be underestimated.
In this model, an estimation procedure for the mean-reverting parameter is proposed in \cite{kluppelberg10}, the pike component being identified via extreme value theory methods.
This estimator which is a modification of the Davis-McCormick estimator,  is proved to be consistent in  \cite{brockwell07}. The same method is used for spike analysis on UK electricity and gas markets in \cite{meyer14}.
Finally, multi-factor models with a Brownian component and a spike component are studied in Meyer and Tankov \cite{meyer08}, Schmidt \cite{schmidt08} and Gonzales {\it et al.} \cite{gonzalez16}. Meyer and Tankov estimate the mean-reverting parameters using spectral methods and the jumps are detected by filtering. In Schmidt \cite{schmidt08}, the parameters of the model are estimated using maximum likelihood with the EM algorithm, implying an approximation of the process with its Euler scheme. 
Gonzales and co-authors develop a Bayesian framework and recover the parameters of the model by MCMC. 
In a more general context than electricity price modelling, Moreno et al.  \cite{moreno11} use a method of moments to estimate the parameters of a jump diffusion model when the log-price is the sum of an arithmetic Brownian motion and a mean reverting compound Poisson process.\\ 

While all these models allow for a good representation of spot prices and spikes, they are not suitable for long-term volatility modelling in the forward prices, which corresponds to the volatility of forward products with delivery far in the future (one quarter to two years). They propose a stationnary modelling of the continuous part of the spot price, which is a mean reverting process or the sum of severals. When time to maturity $T-t$ grows, volatility of $\mathbb{E}(S_T | \mathcal{F}_t)$ goes to zero. To encompass this issue, practitionners use multi-factors model with one factor being non-stationnary, for instance modeled by a Brownian motion. The most used model in practice is the two-factors model, see \cite{schwartz00} for its spot representation and \cite{clewlow99a, clewlow99b, koekebakker05} for the equivalent forward modelling. 
 
\begin{figure}[h!]
    \centering
    \begin{subfigure}[b]{0.4\textwidth}
        \centering
        \includegraphics[width=\textwidth]{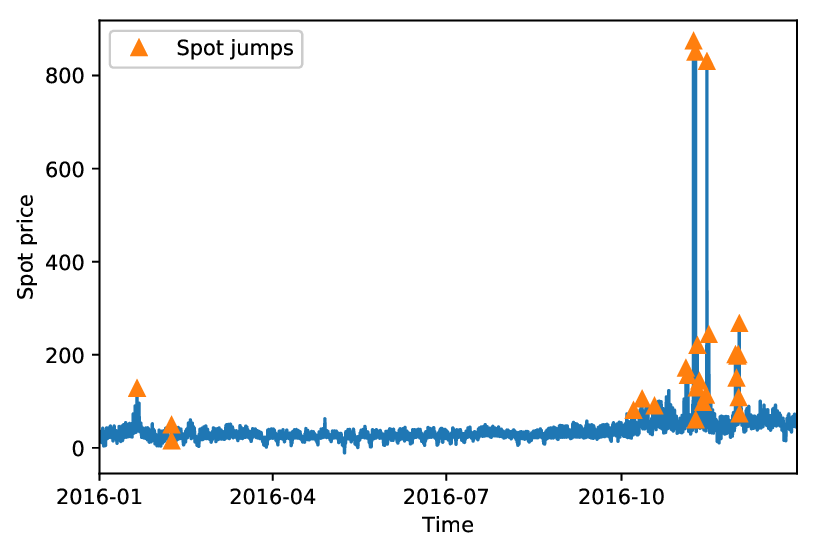}
        \caption{\it French spot price.}
    \end{subfigure}
    \begin{subfigure}[b]{0.4\textwidth}
        \centering
        \includegraphics[width=\textwidth]{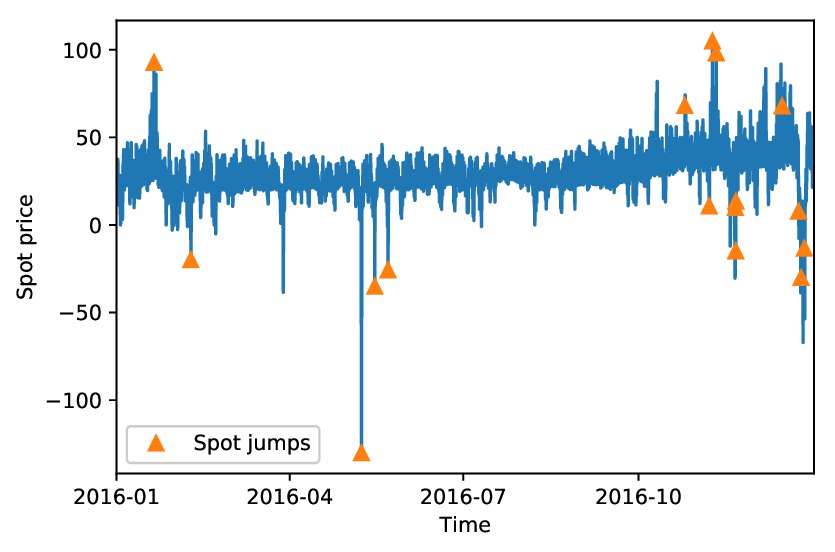}
        \caption{\it German spot price.}
    \end{subfigure}
      
        \begin{subfigure}[b]{0.4\textwidth}
        \centering
        \includegraphics[width=\textwidth]{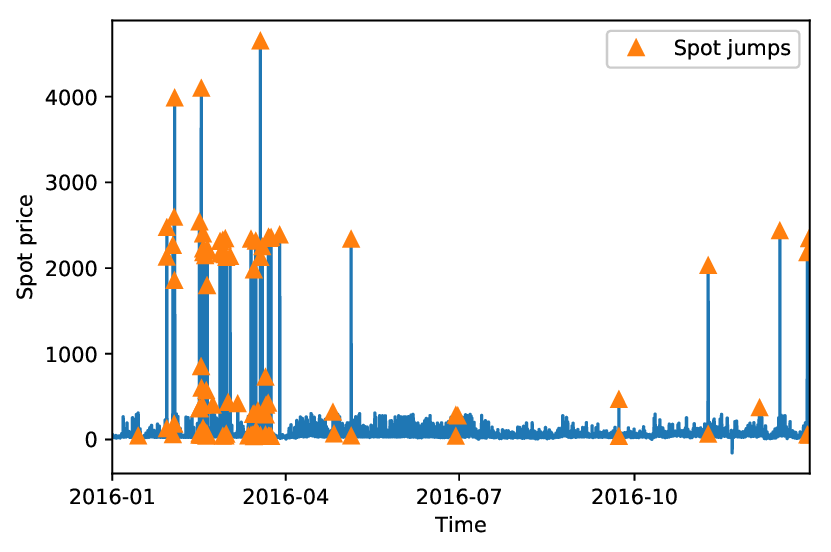}
        \caption{\it Australian spot price.}
    \end{subfigure}
     \caption{\it \label{spotdata} Series of the spot price during the year 2016 for France, Germany and Australia. The frequency of the data is 1 hour for France and Germany and 30 minutes for Australia. Spot jumps are estimated using a threshold of $5\hat{\sigma}\Delta_n^{-0.01}$ where $\hat{\sigma}$ is the multi-power variation of order 20.}
\end{figure}


Most of these papers use calibration procedures to estimate the model parameters, without convergence properties except for the notable exception of \cite{brockwell07,kluppelberg10}. In this paper we propose a theoretical framework of estimation adapted to the models previously cited, i.e. the class of models used for risk management purposes. More precisely, our goal is to construct an estimation procedure for the characteristics of the spike process covering a wide and unifying range of models for the continuous part of spot price (allowing for non-stationarity, stochastic volatility, multiplicative or arithmetic representation and so on).

\subsection{Main results}  \label{pre results}

We consider an extended framework that encompasses \cite{meyer08}, \cite{schmidt08} and \cite{gonzalez16} (when only one spike process if considered). In particular, our approach does not require that the continuous part of the price process is an Ornstein-Uhlenbeck, a necessary condition in the aforementioned models.

\subsubsection*{A semimartingale model with fast mean-reverting jumps} On a rich enough filtered probability space $(\Omega, \mathcal F, ({\mathcal F_t})_{t \geq 0}, \PP)$ that will accommodate all the considered random quantities, we model the electricity spot price $X_t=S_t$ or $X_t=\log S_t$ by 
\begin{equation} \label{our model first part}
X_t = X^c_t + Z_t^{\beta},\;\; t \geq 0,
\end{equation}
where $(X^c_t)_{t \geq _0}$ is a continuous It\^o semimartingale and $(Z_t^{\beta})_{t \geq 0}$ is the so-called {\it  spike process}, governed by a mean-reverting factor $\beta >0$. More specifically, we assume that 
\begin{equation} \label{our model second part}
X^c_t = X_0^c+\int_0^t \mu_s ds + \int_0^t \sigma_s dW_s,\;\;t\geq 0
\end{equation}
where  $(\sigma_t)_{t \geq 0}$ and $(\mu_t)_{t \geq 0}$ are two adapted c\`adl\`ag processes, $(W_t)_{t \geq 0}$ a $(\mathcal F_t)$-standard Brownian motion and 
\begin{equation} \label{our model third part}
Z_t^{\beta} = \int_0^t \int_{\mathbb{R}} x e^{-\beta\left(t-s\right)} \underline{p}\left(ds,dx\right),\;\;t\geq 0,
\end{equation}
with $\underline{p}$ a random Poisson measure on $[0,\infty) \times \mathbb{R}$ independent of $(W_t)_{t \geq 0}$, with intensity 
$$\underline{q} = \lambda \,dt \otimes \nu(dx),$$
for some $\lambda >0$ and a probability measure $\nu(dx)$ on $\R$. We thus model the electricity spot price as a classical continuous It\^o semimartingale $(X_t^c)_{t \geq 0}$ allowing for the usual financial fluctuations and usual models (factor models, mean-reverting models and so on) to which we add a perturbation  $(Z^\beta_t)_{t \geq 0}$ of ``spikes" or ``jumps", triggered by exogenous physical hazard, at intensity $\lambda$ and magnitude $\nu(dx)$, but with a relaxation period $1/\beta$ comparable to $\lambda$ that  
accounts for the absorption of such events by the market toward resulting in stable prices at large scales. The term {\it comparable} is a bit vague at this stage, and will be assessed precisely in Section \ref{les hyp}, enabling us to speak of {\it fast mean-reversion}. In this setting, model \eqref{our model first part}-\eqref{our model second part}-\eqref{our model third part} is well posed and can reproduce, at least {\it visually}, the general shape of electricity spot markets, compare historical data from Figure \ref{spotdata} and sample paths simulations given in Figure \ref{simuX} and detailed in the simulation Section \ref{simulation subsection}.

\begin{figure}[h!]
    \centering
    \begin{subfigure}[b]{0.4\textwidth}
        \centering
       \includegraphics[width=\textwidth]{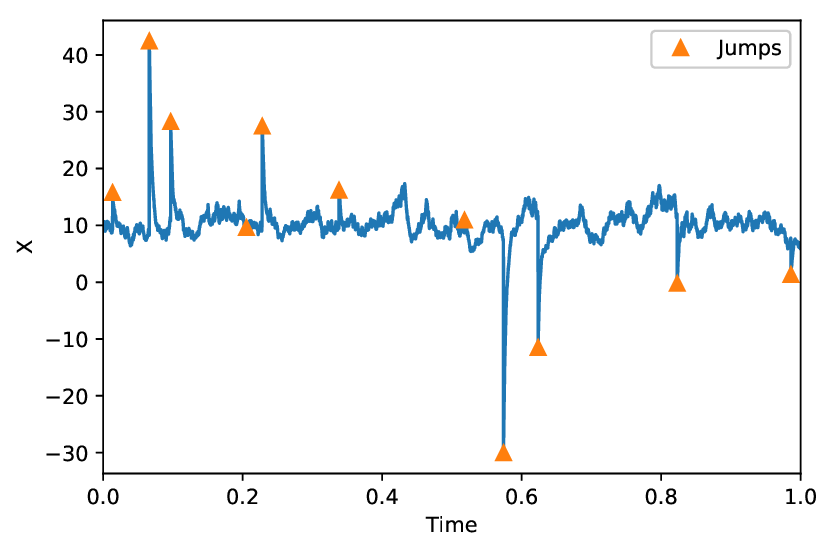}
        \caption{\it $\lambda_n = 10$, $\beta_n = 300$.}
    \end{subfigure}
    \begin{subfigure}[b]{0.4\textwidth}
        \centering
       \includegraphics[width=\textwidth]{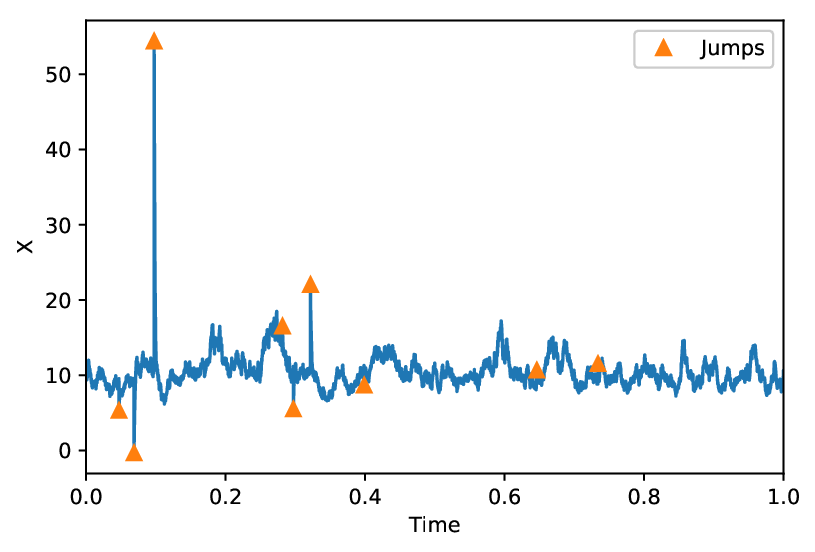}
        \caption{\it $\lambda_n = 10$, $\beta_n = 1000$.}
    \end{subfigure}
      
        \begin{subfigure}[b]{0.4\textwidth}
        \centering
        \includegraphics[width=\textwidth]{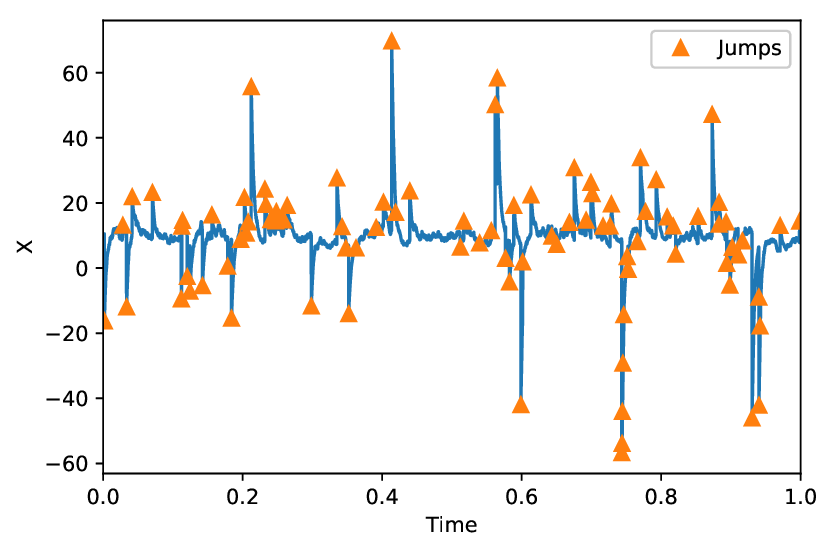}
        \caption{\it $\lambda_n = 75$, $\beta_n = 300$.}
    \end{subfigure}
    
%
    
     \caption{\it \label{simuX}Simulation of the process $X$ in the case of a model having continuous part defined in \eqref{processsimu} and with jump sizes having law $0.4\left(-\mathcal{E}\left(40\right)\right) + 0.6\mathcal{E}\left(30\right)$ for different values of $\lambda_n$ and $\beta_n$.}
\end{figure}

\subsubsection*{Statistical setting}
We assume that we observe the process $(X_t)_{t \geq 0}$ given by \eqref{our model first part}-\eqref{our model second part}-\eqref{our model third part} over the time interval $[0,T]$ on a regular grid 
$$0 =t_{0,n} < t_{1,n} < \ldots, t_{n,n} = T,\;t_{i,n} = i\Delta_n,\;\;\text{for}\;\;0 \leq i \leq n,$$
with mesh $\Delta_n$. Thus we have $n$ (or rather $n+1$) observations
\begin{equation} \label{data}
X^n=(X_0, X_{\Delta_n},\ldots, X_{n\Delta_n}=X_T).
\end{equation}
In the following, for a given process $Y$, we use the classical notations $\Delta_i^n Y = Y_{t_{i,n}} -Y_{t_{i-1,n}}$ and $\Delta Y_{s} = Y_s - Y_{s^-}$. \\
Asymptotics are taken as $n\rightarrow \infty$. We assume that $T$ is constant, and we take $T=1$ with no loss of generality. Equivalently, $\Delta_n = 1/n \rightarrow 0$ as $n \rightarrow \infty$. This asymptotic setting is usually referred to as  the ``high-frequency" framework (for instance the classical textbook \cite{ait14} by A{\"\i}t-Sahalia and Jacod), but this terminology is a bit misleading: our framework certainly belongs to statistical finance, but it has no link to high-frequency finance or microstructure modelling of any sort. In practice, we apply our methodology to three markets: the French EPEX, the German EPEX and the Australian electricity spot in Queensland, see Section \ref{sec:practical implementation} below. We use data between 2015, Jan. 01 and 2016, Dec. 31. with hourly data (even less in the case of Australian data), so that $n=17064$ is considered to be large. Equivalently, one hour is considered to be small in front of 2 years. In our setting, the important fact about the assumption that $T$ is fixed is that we leave out any stationarity or ergodicity of the underlying process. We thus make an implicit statistical robustness assumption, which we believe is of importance when considering recent and changing energy markets over such time horizons.\\

The parameters of interest are $\lambda,\beta>0$ that govern our correction formulas (see the application to forward contracts prices $f(t,T)$ below). In particular we leave out the issue of identifying the continuous semimartingale part $(X_t^c)_{t \geq 0}$ {\it i.e.} the drift $(\mu_t)_{t \geq 0}$ and the volatility process $(\sigma_t)_{t \geq 0}$ as well as the jump distribution $\nu(dx)$ (only an estimation for the moments of $\nu$ is proposed).\\ 

The mean-reversion factor over the observation increment $[t_{i-1,n},t_{i,n}]$ is of size $\beta \Delta_n$, and by requiring $\beta \Delta_n$ to be large  compared to the order of magnitude $\sqrt{\Delta_n}$ of $\Delta_i^n X^c$ , we may hope to recover $\beta$ asymptotically.  We thus introduce the asymptotic setting $\beta = \beta_n$ with the requirement
\begin{equation} \label{first asymptotics}
\beta = \beta_n \rightarrow \infty\;\;\text{while}\;\;\beta_n \sqrt{\Delta_n} \rightarrow \infty.
\end{equation}
Condition \eqref{first asymptotics} becomes
\begin{equation} \label{second asymptotics} \beta_n \sqrt{\Delta_n \lambda_n} \rightarrow \infty
\end{equation}
if we let $\lambda = \lambda_n \rightarrow \infty$. A second crucial assumption is
\begin{equation} \label{third asymptotics}
\beta_n \Delta_n \lesssim 1,
\end{equation} 
since otherwise, the spikes caused by the jumps of $\underline p$ are absorbed by the Brownian fluctuations of $X^c$ due to the fast relaxation period $1/\beta_n$ and therefore cannot be detected by $X^n$. 

\subsubsection*{Statistical results}

Heavily relying on classical techniques in high-frequency finance (for instance \cite[Theorem 10.26, p.374]{ait14}), we estimate in a first step the times and sizes of the jumps which are random quantities, taking into account the interplay between $\beta_n$, $\lambda_n$ and $\Delta_n$ dictated by the asymptotic regime \eqref{first asymptotics}-\eqref{second asymptotics}-\eqref{third asymptotics}, see Proposition \ref{estimatorn}. 
In a second step, we construct an estimator of $\beta_n$
based on an estimator $\widehat{s}_n$ of the right-derivative or {\it instantaneous slope} of $t \mapsto Z^{\beta}_t$ right after a jump is detected. The estimator $\widehat s_n$ is based on averaging of instantaneous slope proxies of the form $\Delta X_{T_q}(1-e^{-\beta_n\Delta_n})$ that govern the relaxation effect after a jump of size $\Delta X_{T_q}$ has occurred at time $T_q$ and it enables us to consider 
$$\widehat{\beta}_n = -\frac{1}{\Delta_n} \log\left(1-\widehat{s}_n\right)$$
as our estimator of $\beta_n$. Since $\beta_n$ itself varies with $n$ and grows to infinity, the notion of convergence has to be considered carefully. Under suitable assumptions, we prove in Theorem \ref{estimatorcj} that the relative error
\begin{equation} \label{conv est beta n}
\mathcal E_n=\frac{\widehat{\beta}_n-\beta_n}{\beta_n} \rightarrow 0
\end{equation}
in probability as $n \rightarrow \infty$.
The error $\mathcal E_n$ has two components: a first term of order
$1/(\beta_n \sqrt{\lambda_n \Delta_n})$ due to Brownian oscillations, and a second term of order $\min\{\lambda_n/\beta_n, 1/\sqrt{\beta_n}\} + \sqrt{\lambda_n}/\beta_n$ that accounts for the effect of jumps that are still present in the price process despite the relaxation effect. When $\beta_n \sqrt{\Delta_n \lambda_n} \rightarrow \infty$ and $\lambda_n/\beta_n \lesssim 1$, we have $\mathcal E_n$ converges to $0$. If we assume further $\sqrt{\beta_n}/\lambda_n \rightarrow 0$, we obtain a central limit theorem for $\mathcal E_n$ with a Gaussian limit and an explicit rate of convergence that depends on the interplay between $\lambda_n,\beta_n$ and $\Delta_n$. We have an analogous result (although less demanding) for the estimation of the jump intensity $\lambda_n$ detailed in Proposition \ref{estimatorlambda}. 


\subsubsection*{Application to pricing forward contracts} We show in Theorem \ref{forward} that in the model for the spot price (with $X = S$) defined by  \eqref{our model first part}-\eqref{our model second part}-\eqref{our model third part}, the price $f(t,T)$ of a forward contract is given by 
$$
f\left(t,T\right) = f^c\left(t,T\right) + f^{\beta}\left(t,T\right),
$$
where
$$
f^c\left(t,T\right) = \mathbb{E}\big[X^c_T\,|\, \mathcal{F}_t\big]
$$
and
$$
f^{\beta}\left(t,T\right) = e^{-\beta\left(T-t\right)}Z_t^{\beta} + \frac{\lambda}{\beta}\int_{\R}x\nu(dx) \left(1-e^{-\beta\left(T-t\right)}\right).
$$
The term $f^c\left(t,T\right)$ corresponds to the price of the forward contract in a continuous case framework. The computation of this value has been extensively studied for different continuous models and it is known analytically for the most common models, see for instance \cite{benth07,benth12} among others. 
The term $f^{\beta}\left(t,T\right)$ is a correction that follows from our approach. It is of order $\lambda/\beta$ and is usually small for the applications we have in mind, see the practical implementation Sections \ref{sec:practical implementation} and \ref{application}. On balance, the presence of spikes does not significantly impact the price of forward contracts to within these order of magnitudes. This is consistent with our data, for which spikes are not observed on forward prices. By neglecting the term $f^{\beta}\left(t,T\right)$, we can calibrate the process $X^c_t$ to the observed forward prices and the process $Z^\beta_t$ to the observed spot prices with the estimation procedure proposed in this paper. The results are similar for the modelling of the logarithm of the spot price ({\it i.e} when $X = \log(S)$) by \eqref{our model first part}-\eqref{our model second part}-\eqref{our model third part}, see Theorem \ref{forwardlog}. We implement the prices of the forward contracts and Call options from a model calibrated to historical data on electricity prices in Section \ref{application} and we show the impact of the spike modelling on the valuation of a strip of options with payoff of the form $\sum_{i=1}^I \left(S_{t_{i,n}} - K\right)^+$ for different times $t_{i,n}$. As expected, the value of this option increases if we add significant large spikes. In particular, we show that some out-of-the-money options have a significant value if we incorporate spikes in our modelling, while having a value close to $0$ otherwise. 


\subsection{Organisation of the paper}
Section \ref{mainresults} develops a rigorous mathematical framework for the stochastic model \eqref{our model first part}-\eqref{our model second part}-\eqref{our model third part} and gives the explicit construction of the estimators described in Section \ref{pre results} above together with their asymptotic properties in Propositions \ref{estimatorn}, \ref{estimatorlambda} and Theorem \ref{estimatorcj}. Section \ref{sec:practical implementation} establishes the numerical feasibility and consistency of our statistical estimation results on simulated and real data, based over two years (2015 and 2016) of electricity spot prices in three different markets (French, German and Australian). Section \ref{application} is devoted to the application of our model and statistical results to forward contracts. We establish in Theorem \ref{forward} a correction formula to get analytical forward prices and study the valuation of a strip of European Call options. The proofs are given in Section \ref{proofs}. 

\section{Statistical results}
\label{mainresults}

\subsection{Model assumptions} \label{les hyp}
We consider the process $(X_t)_{t \geq 0}$ defined by \eqref{our model first part}-\eqref{our model second part}-\eqref{our model third part} in Section \ref{pre results}. Following closely the standard notation of A{\"\i}t-Sahalia and Jacod \cite{ait14}, if $\nu$ is a positive measure and $f$ a $\nu$-integrable function, we write $f\left(x\right) \star \nu = \int_{\mathbb{R}} f\left(x\right) \nu\left(dx\right)$.
Remember also that we work over a finite time horizon $T=1$. 

\begin{assumption} \label{assumpsigma} 
We have $\E[(X_t^{c})^2]<\infty$ for every $t\geq 0$. Moreover, $t \mapsto \sigma_t$ is continuous on $[0,1]$ and for some deterministic  
$\underline{\sigma}, \bar{\sigma}, c_0>0$, we have $0 < \underline{\sigma}^2 \leq \inf_{t}\sigma_t^2  \leq \sup_t\sigma_t^2 \leq \bar{\sigma}^2$, $\sup_t|\mu_t| \leq c_0$, 
$\nu(\{0\})=0$ and $ |x|^2 \star \nu < \infty$.
\end{assumption}

Since our asymptotic results will be given in distribution (see Theorem \ref{estimatorcj} below), the conditions on the drift $(\mu_t)_{t\geq 0}$ and $(\sigma_t)_{t \geq 0}$ can substantially be weakened (in order to accommodate for instance diffusion coefficients of the form $\sigma_t = X_t^ch(X_t^c)$ with a bounded $h$ or even locally integrable) by standard localisation procedures, see for instance \cite[Section 4.4.1]{jacod11}. We observe \eqref{data}
and asymptotics are taken as $n \rightarrow \infty$ or equivalently $\Delta_n=n^{-1}\rightarrow 0$. 
We consider different asymptotic regimes for  $\lambda=\lambda_n$ and $\beta = \beta_n$, compare Equations \eqref{first asymptotics}-\eqref{second asymptotics}-\eqref{third asymptotics} and the accompanying discussion in Section \ref{pre results} above.

\begin{assumption} \label{betalambda} We have
$$\liminf_n \lambda_n > 0, \;\;\lambda_n \lesssim \beta_n,\;\;\beta_n \Delta_n \lesssim 1,\;\;\text{and}\;\;\lambda_n \Delta_n \rightarrow 0.$$
\end{assumption}
The condition $\lambda_n \lesssim \beta_n$ ensures the stability of $X_t$ as $n \rightarrow \infty$ since
$\mathrm{Var}(X_t)= \mathrm{Var}\left(X_t^c\right)+ \mathrm{Var}(Z_t^{\beta_n}) = \mathrm{Var}\left(X_t^c\right) + |x|^2 \star \nu \,\frac{\lambda_nt}{2\beta_n}\left(1-e^{-2\beta_n}\right) \rightarrow \infty$
if $\sup_n\lambda_n/\beta_n=\infty$. The condition $\beta_n \Delta_n \lesssim 1$ is necessary to identify spikes (or jumps): otherwise, a spike that occurs in the interval $\left(\left(i-1\right)\Delta_n, i\Delta_n\right]$ will be absorbed by the relaxation effect before we observe $X_{i\Delta_n}$. Finally, the condition $\lambda_n \Delta_n \rightarrow 0$ controls the no accumulation of jumps within the rate of observation.\\

In order to estimate the times and sizes of the jumps, we need the following:

\begin{assumption} \label{assumpestimatorn} We have either $(I)$  or $(II)$, where
\begin{enumerate}
\item[(I)] For some $\varpi \in \left(0,1/2\right)$:
$$\mathrm{(i)}\;\lambda_n^2 \Delta_n \rightarrow 0,\;\;\mathrm{(ii)}\;\lambda_n {\bf 1}_{\{|x| > \Delta_n^{1/2-\varpi}/(\beta_n \Delta_n)\}} \star \nu \rightarrow 0,\;\;\text{and}\;\;\mathrm{(iii)}\;\lambda_n {\bf 1}_{\{|x| < \Delta_n^{1/2-\varpi}\}} \star \nu \rightarrow 0.$$
\item[(II)]  For some $\varpi \in \left(0,1/2\right)$ and a sequence of integers $k_n \geq 1$:
$$\mathrm{(i)}\;\lambda_n^2 \Delta_n k^2_n \rightarrow 0,\;\;\mathrm{(ii)}\;\;\lambda_n {\bf 1}_{\{|x| > e^{\beta_n \Delta_n k_n} \Delta_n^{1/2-\varpi}\}} \star \nu \rightarrow 0,\;\;\mathrm{(iii)}\;\lambda_n {\bf 1}_{\{|x| < \Delta_n^{1/2-\varpi}\}} \star \nu \rightarrow 0$$
and $\mathrm{(iv)}\;\lambda_n e^{-(\beta_n \Delta_n^{1-\varpi})^2} \rightarrow 0$.
\end{enumerate}
\end{assumption}

Assumption \ref{assumpestimatorn} (I) implies $\beta_n \Delta_n = o\left(1\right)$. Condition (i) ensures that the number of jumps in an interval of size $\Delta_n$ is essentially $1$. In the case where $\beta_n$ is bounded, (ii) is implied by (i) and we have the usual conditions for the detection of jumps (see Mancini \cite{mancini04}). Condition (ii) 
controls the size of the mean-reversion.
Condition (iii) controls the size of the small jumps that cannot converge too fast to 0. If the jumps are bounded below by some constant as Mancini \cite{mancini04} the condition is automatically satisfied.

\smallskip
Assumption \ref{assumpestimatorn} (II) (iv) implies that $\beta_n \Delta_n^{1-\varpi}\rightarrow \infty$ and in particular $\beta_n \Delta_n^{1/2} \rightarrow \infty$, implying that the mean reversion of order $\beta_n \Delta_n$ is stronger than the order of magnitude $\Delta_n^{1/2}$ of Brownian increments.
It also allows for the case $\beta_n \Delta_n \approx 1$. In the setting of  Assumption \ref{assumpestimatorn} (II), the mean reversion is more difficult to distinguish from the jumps and in the case $\beta_n \Delta_n \approx 1$, the jumps and the drift have the same size and are not distinguishable from their size solely. Condition (i) states that there is at most one jump in an interval of size $k_n$ which is large enough for the spike to vanish. Assumption \ref{assumpestimatorn} (II) also implies that $\lambda_n^2/\beta_n \rightarrow 0$. Assumption \ref{assumpestimatorn} (II) allows for high values of $\beta_n$ but the number of jumps needs then to be smaller than in case of Assumption \ref{assumpestimatorn} (I) compared to $\beta_n$. 

\subsection{Estimation of the jump times and $\lambda_n$}
\label{sec_estimation_jumps}
We first construct estimators of the sequence of the jumps and of their intensity $\lambda_n$. Assumption \ref{assumpestimatorn} is in force. Let
$$N_t = \sum_{s \leq t} {\bf 1}_{\{\Delta X_s \neq 0\}},\;\;t\geq 0,$$
denote the number of jumps of $(X_t)_{t \geq 0}$ up to time $t$ and let
$T_1 < T_2 < \cdots < T_q < \cdots$ 
denote the random times at which jumps occur.
By construction, the sizes of jumps $\left(\Delta X_{T_q}\right)_{q \geq 1}$ form a sequence of independent and identically distributed random variables independent of $(N_t)_{t \geq 0}$. 
Let $i(n,q)$ be the random integer such that 
\[\left(i\left(n,q\right)-1\right)\Delta_n < T_q \leq i\left(n,q\right)\Delta_n.\] 
We introduce the increasing sequence
$$1 \leq \mathcal{I}_n\left(1\right) <  \mathcal{I}_n(2) < \ldots  <  \mathcal I_n\big(n_{\max}\big) \leq n$$
of indices $i \in \{1, \ldots, n\}$ defined by the realisation of the following successive events: 
$$\Big\{\frac{|\Delta_i^n X|}{\sqrt{\Delta_n}} > v_n \Big\}\;\;\text{under Assumption \ref{assumpestimatorn}\, (I)}$$
and
$$\Big\{\frac{|\Delta_i^n X|}{\sqrt{\Delta_n}} > v_n, \; \Delta_i^n X \Delta_{i+1}^n X < 0\Big\}\;\;\text{under Assumption \ref{assumpestimatorn}\,(II)}.$$
with $v_n \asymp \Delta_n^{-\varpi}$, following the notations of \cite[Ch. 10]{ait14}. The $\mathcal I_n(k)$ define a subset of $\{1,\ldots, n\}$ that define in turn an estimator of $\lambda$ via
$$\widehat \lambda_n = \mathrm{Card}\big\{\mathcal{I}_n(1), \ldots, \mathcal I_n(n_{\max})\big\} = n_{\max}.$$ 

Under Assumption \ref{assumpestimatorn} (II), we need the supplementary condition $\Delta_i^n X \Delta_{i+1}^n X < 0$ for the following reason: whenever a jump occurs, the mean reverting is dominant in the next observation interval and has a direction opposite to the sign of the jump. Furthermore, it enables us to discard the increments caused by the mean reversion that are large enough to be detected as jumps. Indeed, if we detect a false jump due to the mean reversion effect, the next increments will follow the same dynamics and it will share the same sign with first increment. 
The property that no jump lies within the next observation interval is ensured by the existence of $k_n$ defined in Assumption  \ref{assumpestimatorn} (II).

\begin{remark} \label{rk: two algo}
Note that $\widehat \lambda_n$ is actually defined with two different algorithms, whether Assumption \ref{assumpestimatorn} (I) or  \ref{assumpestimatorn} (II) is taken under consideration. When stating the results, we refer to the same notation $\widehat \lambda_n$ that actually corresponds to two algorithms, which we will refer to as Algorithm (I ) and Algorithm (II) in the proofs and when processing the data, under Assumption \ref{assumpestimatorn} (I) and  \ref{assumpestimatorn} (II) respectively.
\end{remark}

 Let 
\[
\Omega_n = \Big\{\widehat{\lambda}_n = N_1,\forall q \in \{1, ...,N_1\}: T_q \in \big(\mathcal{I}_n(q)\Delta_n-\Delta_n,\mathcal{I}_n(q)\Delta_n\big]\Big\}. 
\]
\begin{proposition} \label{estimatorn} 
Work under Assumptions \ref{assumpsigma}, \ref{betalambda} and \ref{assumpestimatorn}.
We have $\mathbb{P}\left(\Omega_n\right) \rightarrow 1$.
\end{proposition}

The proof of Proposition \ref{estimatorn} relies on a result of  A{\"\i}t-Sahalia and Jacod \cite[Theorem 10.26, p.374]{ait14}. However, the presence of a drift term $-\beta_n \int_{0}^t Z_s^{\beta_n} ds$ that depends on $n$ together with the fact that $\lambda_n \rightarrow \infty$ makes the extension not trivial.
Proposition \ref{estimatorn} also provides us with an estimator $\widehat \lambda_n$ of $\lambda_n$. In the case $\lambda_n \rightarrow \infty$, we have the following asymptotic property:

\begin{proposition} \label{estimatorlambda}
Work under Assumption \ref{assumpsigma}, \ref{betalambda} and \ref{assumpestimatorn} and assume that $\lambda_n \rightarrow \infty$. We have
\begin{equation} \label{TCL lambda}
\sqrt{\lambda_n}\frac{\widehat{\lambda}_n - \lambda_n}{\lambda_n} \rightarrow \mathcal N(0,1)
\end{equation}
in distribution as $\lambda_n \rightarrow \infty$. 
\end{proposition}
This result is straightforward: on the event $\Omega_n$, we have $\widehat \lambda_n = N_1$, {\it i.e.} $\widehat \lambda_n$ is a Poisson random variable with parameter $\lambda_n$, for which \eqref{TCL lambda} holds. We conclude thanks to  $\mathbb{P}\left(\Omega_n\right) \rightarrow 1$. 

\begin{remark}
As for the optimality of the result, consider indeed the seemingly richer experiment where one continuously observes a Poisson process $(\mathcal P_t)_{0 \leq t \leq 1}$ with intensity $\lambda >0$. The variable $\mathcal P_1$ is a sufficient statistic and the Cramer-Rao bound tells us that any unbiased estimator $\widehat \lambda$ necessarily satisfies $\E[(\widehat \lambda-\lambda)^2] \geq I(\lambda)^{-1}$, where $I(\lambda) = 1+\lambda^{-1}$ is the Fisher information associated to the observation of $\mathcal P_1$, {\it i.e.} a Poisson random variable with parameter $\lambda$. Equivalently $\E\big[\big(\frac{\widehat \lambda-\lambda}{\lambda}\big)^2\big] \geq (\lambda^{-1}+\lambda^{-2}) \sim \lambda^{-1}$ as $\lambda \rightarrow \infty$ which is consistent with the convergence \eqref{TCL lambda} for fixed $\lambda$.
\end{remark}

\begin{remark}
The assumption $\lambda \rightarrow \infty$ may be acceptable for some electricity markets revealing a lot of extreme events, like for example in Australia. However, it can be questionable for electricity markets like in France or Germany, although the latter is currently revealing a lot of negative prices.
\end{remark}

\medskip
A natural estimator of the jump sizes is $\Delta^n_{\mathcal{I}_n\left(q\right)} X$ for $q \in \{1,...,\widehat{\lambda}_n\}$, see \cite[Theorem 10.21, p.370]{ait14}. In our case, $\Delta^n_{\mathcal{I}_n\left(q\right)} X$ is equal to $\Delta X_{T_q} e^{-\beta_n\left(T_q-\mathcal{I}_n\left(q\right)\Delta_n\right)}$ plus a negligible term. If $\beta_n \Delta_n \to 0$, $\Delta^n_{\mathcal{I}_n\left(q\right)} X$ is then equivalent to $\Delta X_{T_q}$ but if $\beta_n \Delta_n \asymp 1$, the bias $\Delta X_{T_q}\left(1- e^{-\beta_n\left(T_q-\mathcal{I}_n\left(q\right)\Delta_n\right)}\right)$ remains and it is not possible to identify the size of the jump. However, if $\lambda_n \to \infty$, one can infer some statistical properties of the jumps size. 
We have the following result:

\begin{proposition} \label{estimatornu} Work under Assumption \ref{assumpsigma}, \ref{betalambda}, \ref{assumpestimatorn} and assume that $\lambda_n \rightarrow \infty$. Let $m \geq 1 \in \mathbb{N}$ such that $|x|^{2m} \star \nu < \infty$,
\[\frac{m\beta_n \Delta_n }{\left(1-e^{-m\beta_n\Delta_n}\right)\widehat{\lambda}_n} \sum_{q=1}^{\widehat{\lambda}_n} \left(\Delta_{\mathcal{I}_n\left(q\right)} X\right)^m \rightarrow
x^m \star \nu 
\]
in probability as $\lambda_n \rightarrow \infty$. Furthermore, if $|x|^{3m} \star \nu < \infty$, 
\[\frac{\sqrt{\lambda_n}m\beta_n \Delta_n\left(1-e^{-2m\beta_n\Delta_n}\right)}{2(1-e^{-m\beta_n\Delta_n})^2}\left(\frac{m\beta_n \Delta_n }{\left(1-e^{-m\beta_n\Delta_n}\right)\widehat{\lambda}_n} \sum_{q=1}^{\widehat{\lambda}_n} \left(\Delta_{\mathcal{I}_n\left(q\right)} X\right)^m -
x^m \star \nu\right) \rightarrow \mathcal{N}(0,1) 
\]
in distribution as $\lambda_n \rightarrow \infty$.

\end{proposition}
Proof of Proposition \ref{estimatornu} is immediate: it only requires the computation of $\int_0^1 e^{-\beta_n(t-\lfloor \frac{t}{\Delta_n}\rfloor\Delta_n}dt = \frac{1-e^{-\beta_n\Delta_n}}{\beta_n \Delta_n}$ and the use of \cite[Theorem 3]{peccati08} for the central limit theorem. Also, the proof for the case $m = 1$ appears in the proof of Theorem \ref{estimatorcj}. We then omit the proof. Combining this result with our estimator of $\beta_n$ provided below enables us to have an estimator for the moments of $\nu$.

\subsection{Estimation of \texorpdfstring{$\beta_n$}{beta}}
We are ready to construct an estimator of $\beta_n$, based on the estimation of the right-derivative or {\it instantaneous slope} of $t \mapsto Z^{\beta}_t$ right after a jump is detected. The estimator of this right-derivative, say $s_n$, is  based on averaging of instantaneous slope proxies of the form $\Delta X_{T_q}(1-e^{-\beta_n\Delta_n})$ that govern the relaxation effect after a jump of size $\Delta X_{T_q}$ has occurred at time $T_q$ and it enables us to consider 
$\widehat{\beta}_n = -\frac{1}{\Delta_n} \log\left(1-\widehat{s}_n\right)$
as our estimator of $\beta_n$. 

More precisely, let $\mathrm{\mathrm{sgn}}(x)=1$ if $x \geq 0$ and $-1$ otherwise.
On the event $\{\widehat \lambda_n >0\}$, define $\widehat \beta_n$ via 
\begin{equation} \label{construct estim beta}
\exp(-\Delta_n\widehat \beta_n) = \max\Big\{1+\frac{\sum_{q = 1}^{\widehat{\lambda}_n}\mathrm{\mathrm{sgn}}(\Delta^n_{\mathcal{I}_n(q)} X) \big(\Delta^n_{\mathcal{I}_n(q)+1} X + 2\Delta_n\sum_{j=1}^{q-1} \Delta^n_{\mathcal{I}_n(j)} X\big)}{\sum_{q = 1}^{\widehat{\lambda}_n} |\Delta^n_{\mathcal{I}_n(q)} X|}, \Delta_n\Big\}
\end{equation}
and set $\widehat \beta_n=0$ otherwise. Our main result describes precisely the behaviour of $\widehat \beta_n$ under the different asymptotic regimes of interest.

\begin{remark}
The same comments for $\widehat \beta_n$ as in Remark \ref{rk: two algo} for $\widehat \lambda_n$ hold here.
\end{remark}

\begin{theorem} \label{estimatorcj} Work under Assumptions \ref{assumpsigma}, \ref{betalambda} and \ref{assumpestimatorn}. Let $\beta_n \sqrt{\lambda_n \Delta_n} \rightarrow \infty$. On the set $\{\widehat{\lambda}_n > 0\}$, we have
$$\frac{\big|\widehat {\beta}_n - \beta_n\big|}{\beta_n} \lesssim \lambda_n\Delta_n+ \frac{1}{\beta_n\sqrt{\lambda_n\Delta_n}} + \min\left\{\frac{1}{\sqrt{\beta_n}},\frac{\lambda_n}{\beta_n}\right\}
$$
in probability. More precisely, 
\begin{enumerate}
\item[i)] On $\{\widehat{\lambda}_n > 0\}$, we have
\[\frac{\widehat {\beta}_n - \beta_n}{\beta_n} = \mathcal M_n + \mathcal V_n \mathcal J_n^T,\]
where
\[
\mathcal M_{n} = e^{\beta_n \Delta_n} \frac{\lambda_n}{\beta_n} \frac{(x \star \nu) \; (\mathrm{\mathrm{sgn}}(x) \star \nu)}{|x| \star \nu} \big(\frac{e^{\beta_n \Delta_n}-1}{\beta_n \Delta_n} -1\big),\]
$\mathcal V_{n} = ( \mathcal V_n^{(i)})_{1 \leq i \leq 4}
\in \R^4$ is such that
$$
\left\{
    \begin{array}{ll}
    \mathcal V_{n}^{(1)}&=e^{\beta_n \Delta_n}\frac{\sqrt{\lambda_n}}{\sqrt{3} \beta_n |x| \star \nu} \big((\mathrm{\mathrm{sgn}}(x) \star \nu )^2 (|x|^2 \star \nu) + ( x \star \nu )^2 - 2 (\mathrm{\mathrm{sgn}}(x) \star \nu) (|x|^2 \star \nu)\big)^{1/2}, \\ \\
    	 \mathcal V_{n}^{(2)}  &= e^{\beta_n \Delta_n}\min\big\{\big(\tfrac{|x|^2 \star \nu}{(|x| \star \nu )^2} \frac{1}{2\beta_n}\frac{(1-e^{-2\beta_n \Delta_n})}{2\beta_n \Delta_n}\big)^{1/2},\frac{\lambda_n}{\beta_n}\big\}, \\ \\
        \mathcal V_{n}^{(3)} &=e^{\beta_n \Delta_n}\big(\frac{\beta_n \Delta_n}{1-e^{-\beta_n \Delta_n}}\big)\frac{\sqrt{\int_0^1 \sigma_s^2 ds}}{|x| \star \nu \sqrt{ \lambda_n}\beta_n \sqrt{\Delta_n}},\\ \\
   	  \mathcal V_{n}^{(4)}  &=  e^{\beta_n \Delta_n}\frac{\sqrt{\int_0^1 \sigma_s^2 ds} \sqrt{\Delta_n}}{|x| \star \nu \sqrt{\lambda_n}},
    \end{array}
\right.
$$    
 and
$\mathcal J_{n} = ( \mathcal J_n^{(i)})_{1 \leq i \leq 4} \in \R^4$ is bounded in probability as $n \rightarrow \infty$. 
\item[ii)] If $\lambda_n \rightarrow \infty$, then 
$$
(\mathcal J_{n}^{(3)},\mathcal J_{n}^{(4)}) \rightarrow \mathcal{N}(0,\mathrm{Id}_{\R^2})
$$
in distribution as $n \rightarrow \infty$.  
\item[iii)] If $\lambda_n \rightarrow \infty$, $|x|^3 \star \nu < \infty$ and 
$
( \mathrm{\mathrm{sgn}}(x) \star \nu)^2 |x|^2 \star \nu + ( x \star \nu )^2 - 2 \mathrm{\mathrm{sgn}}(x) \star \nu |x|^2 \star \nu\neq 0,
$
we have
$$(\mathcal J_{n}^{(1)},\mathcal J_{n}^{(3)},\mathcal J_{n}^{(4)}) \rightarrow \mathcal{N}(0,\mathrm{Id}_{\R^3})$$
in distribution as $n \rightarrow \infty$.  
\item[iv)] If  $\beta_n/\lambda_n^2 \rightarrow 0$, the conditions of $\mathrm{iii)}$ and $ |x|^4 \star \nu < \infty$ hold together, we finally obtain
$$
\mathcal J_{n}  \rightarrow \mathcal{N}(0,\mathrm{Id}_{\R^4})
$$
in distribution as $n \rightarrow \infty$.
\end{enumerate}
\end{theorem}

Some remarks on the different error terms: {\bf 1)} the term of order $1/(\beta_n \sqrt{\lambda_n \Delta_n})$ 
accounts for the presence of a Brownian motion in the term $(X^c_t)_{t \geq 0}$. When $\lambda_n$ is bounded, we need $\beta_n \sqrt{\Delta_n} \rightarrow \infty$ or equivalently $\sqrt{\Delta_n} = o\left(\beta_n \Delta_n\right)$: the size of the slope of $(Z^{\beta}_t)_{t \geq 0}$ after a jump needs to dominate the Brownian motion part which is of order $\sqrt{\Delta_n}$. In the case where $\lambda_n \rightarrow \infty$, we can average the error due to the Brownian martingale part and then diminish the order of the error. In that case, we do not need the restriction $\sqrt{\Delta_n} = o\left(\beta_n \Delta_n\right)$ anymore but rather $\sqrt{\Delta_n/\lambda_n} = o\left(\beta_n \Delta_n\right)$.  {\bf 2)}
The error terms of order $\min\{\frac{1}{\sqrt{\beta_n}}, \frac{\lambda_n}{\beta_n}\}$, $\frac{\sqrt{\lambda_n}}{\beta_n}$ and $\lambda_n \Delta_n$ account for the jumps that occur before the observation increment used to estimate the slope of the process. {\bf 3)} The term 
$2\Delta_n\sum_{j=1}^{q-1} \Delta^n_{\mathcal{I}_n(j)} X$ introduced in the definition of $\widehat \beta_n$ in \eqref{construct estim beta}
is a bias correction that enables us to obtain a consistent estimator in the case $\lambda_n/\beta_n \approx 1$.

%
%
%
%

\section{Practical implementation}

\label{sec:practical implementation}

\subsection{Choice of the threshold $v_n$} The method to detect the jumps is based on the classical use of the threshold which is proportional to $\Delta_n^{-\varpi}$. As for  the choice of the threshold and $\varpi$, no exact method is provided in the literature. However, the threshold is recommended to be chosen of the form $v_n=C \hat{\sigma} \Delta_n^{-\varpi}$  in \cite[Section 5.3]{ait09} and \cite[Section 6.2.2, p. 187]{ait14} where $C$ is a constant and $\hat{\sigma}$ is an estimator of the integrated volatility, defined by  $(\int_0^1 \sigma_s^2 ds)^{1/2}$ (actually, the square root of the integrated square volatility). A popular rule-of-thumb consists in picking $\varpi$ close to $0$. Moreover, \cite{ait09} suggests to choose $C$ between 3 and 5. 

\medskip
It remains to find an estimator of the integrated volatility $(\int_0^1 \sigma_s^2ds)^{1/2}$. A natural choice is the multipower variation estimator defined by 
\[
\hat{\sigma} = \sqrt{\Big(\frac{\Gamma(\frac{1}{2})}{2^{\frac{1}{k}}\Gamma(\frac{1}{2} +\frac{1}{k})}\Big)^{k} \sum_{i=1}^{n-k} |\Delta_i^n X \Delta_{i+1}^n X \ldots \Delta_{i+k}^n X|^{\frac{2}{k}} }
\] with $k$ the order of the estimator, see \cite{barndorff06} and \cite{veraart10} for more details. The order of the multipower variation estimator is set to 20 in the practical applications in this paper, which is high compared to the orders typically chosen in the literature. This choice is justified by the strong mean reversion of the spikes. As for the jumps, spikes have large increments that need to be compensated in the multipower variation estimator and can be present during two or three time steps. We compensate these large increments with a higher order of the multipower variation. Some simulations on simple models show that an order of 20 looks reasonable.

\subsection{Numerical illustration} \label{simulation subsection}
In this section, we study the performances of our estimation procedures on simulated data of the process defined by \eqref{our model first part}-\eqref{our model second part}-\eqref{our model third part}. We tested a wide range of values for $(\lambda_n,\beta_n)$ in order to illustrate the results of Theorem \ref{estimatorcj}. To be consistent with real data, the process is simulated with a step time $\Delta_n = 10^{-4}$, which is the order of magnitude corresponding to one year of hourly data observations. We pick
\begin{equation} \label{processsimu}
dX^c_t = X^c_t\big((2 - 100\log(X^c_t))dt + 2dW_t\big)
\end{equation}
corresponding to the exponential of an Ornstein-Uhlenbeck process with the mean reverting parameter equal to 100 and volatility parameter equal to 2. The sizes of the jumps follow the law $0.4\left(-\mathcal{E}\left(15\right)\right) + 0.6\mathcal{E}\left(10\right)$, with $\mathcal E(\rho)$ denoting the exponential distribution with parameter $\rho>0$. Figure \ref{simuX} illustrates a sample path of the process for different parameters $\lambda_n$ and $\beta_n$. We realize 10000 simulations. We use jump detection using Algorithm (I), corresponding only to the use of a threshold,  as well as Algorithm (II), corresponding to filtering the previous jumps by keeping only increments that have successive opposite signs, respectively working theoretically in the case of \ref{assumpestimatorn}\,(I) and \ref{assumpestimatorn}\,(II).

Table \ref{tab_simuC3}, \ref{tab_simuC4} and \ref{tab_simuC5} show the different results of estimated $(\widehat{\lambda}_n,\widehat{\beta}_n)$ using a threshold equal to $C \hat{\sigma}\Delta_n^{-0.01}$, for $C = 3, \; 4, \; 5$. As expected, for large $\beta_n$, the estimation on filtered jumps gives satisfactory results of $\widehat{\beta}_n$, whatever the value of $\lambda_n$ satisfying the different assumptions, and seems to slightly underestimate $\lambda_n$, whereas the estimation under Algorithm (I) gives bad results on $\widehat{\beta}_n$ and overestimates $\lambda_n$. 

The results for $\beta_n=20$ highlights, as expected, the need in Assumption \ref{assumpestimatorn}\,(II) to have a small number of jumps ($\lambda_n=10$) to get satisfactory results.  
Because of the expected $\beta_n$ on real data, we will focus, in the next section, on the estimation procedure on filtered jumps ({\it i.e.} using Algorithm (II). This choice is also justified by the fact that underestimating the numbers of spikes seems to have a lower impact on the quality of $\widehat{\beta}_n$ than overestimating them. One also observes that the choice of $C$ has an impact on the estimation of $\lambda_n$: the threshold to select jumps increases with $C$ and then the estimated $\lambda_n$ which is the number of jumps decreases. On the contrary, it has a low impact on the estimation of $\beta_n$, except in the case $C = 3$ for small $\lambda_n$ where the number of spikes is overestimated, leading to a bad estimator for $\beta_n$. Finally, the impact of the order of the multipower estimator of the volatility is illustrated in Figure \ref{impactmultipower}. The impact is negligible from an order of $20$, where the estimators mean of $\lambda_n$ and $\beta_n$ remain constant and give a good estimation. Before 20, it has only an impact on the estimation of $\lambda_n$: the spike effect has not disappeared in the estimator and $\hat{\sigma}$ is overestimated, inducing less detected spikes and an underestimation for $\lambda_n$. This supports the choice of a high order for the multipower estimator.

\begin{table}[h!]
	\centering
	\begin{tabular}{|c|c|c|c|c|c|c|c|c|}
		\hline
		\multirow{3}{*}{$\left(\lambda_n, \beta_n\right)$}	&
		\multicolumn{4}{c|}{Algorithm (I)} & \multicolumn{4}{c|}{Algorithm (II)} \\
		\cline{2-9}
		 &\multicolumn{2}{c|}{$\widehat{\lambda}_n$} &\multicolumn{2}{c|}{$\widehat{\beta}_n$} &
		\multicolumn{2}{c|}{$\widehat{\lambda}_n$} &\multicolumn{2}{c|}{$\widehat{\beta}_n$} \\
		\cline{2-9}
		& Mean & C. I.   & Mean & C. I. 
		& Mean & C. I.   & Mean & C. I.  \\
		\hline
     $(10, 2)$ &   27.0 &    $\left[18, 36\right]$ &     4.4 &       $\left[-16.6, 25.8\right]$ &  13.3 &   $\left[8, 20\right]$ &     86.5 &       $\left[13.6, 179.6\right]$ \\
    $(10, 20)$ &   21.3 &    $\left[11, 35\right]$ &   -31.6 &      $\left[-152.2, 23.7\right]$ &   8.9 &   $\left[5, 14\right]$ &     37.1 &        $\left[18.1, 66.8\right]$ \\
   $(10, 200)$ &   95.1 &   $\left[70, 117\right]$ &  -592.6 &     $\left[-1839.0, 28.7\right]$ &   9.4 &   $\left[5, 15\right]$ &    204.5 &      $\left[186.5, 225.7\right]$ \\ 
   \hline
  $(10, 2000)$ &   94.1 &   $\left[54, 137\right]$ & -4430.1 &  $\left[-5860.2, -2271.0\right]$ &   9.7 &   $\left[5, 15\right]$ &   2003.0 &    $\left[1978.9, 2026.2\right]$ \\
 $(10, 20000)$ &   45.7 &    $\left[31, 62\right]$ &  -727.9 &    $\left[-1262.9, -77.7\right]$ &  18.0 &  $\left[12, 25\right]$ &  19152.0 &  $\left[17840.8, 19845.5\right]$ \\
     $(75, 2)$ &   79.6 &    $\left[66, 94\right]$ &     2.2 &         $\left[-2.3, 6.6\right]$ &  42.1 &  $\left[32, 53\right]$ &     42.5 &       $\left[-3.5, 232.1\right]$ \\
  \hline
    $(75, 20)$ &   73.0 &    $\left[59, 88\right]$ &    19.7 &        $\left[12.4, 27.8\right]$ &  49.5 &  $\left[39, 60\right]$ &     50.6 &       $\left[24.4, 158.9\right]$ \\
   $(75, 200)$ &  106.7 &   $\left[82, 132\right]$ &   147.0 &       $\left[94.5, 190.7\right]$ &  63.3 &  $\left[52, 75\right]$ &    226.2 &      $\left[198.8, 289.4\right]$ \\ 
  $(75, 2000)$ &  351.3 &  $\left[290, 411\right]$ & -2483.0 &  $\left[-3471.0, -1566.8\right]$ &  69.4 &  $\left[56, 83\right]$ &   2021.5 &    $\left[1959.7, 2116.0\right]$ \\
\hline
 $(75, 20000)$ &  229.2 &  $\left[189, 271\right]$ & -1158.3 &  $\left[-1291.5, -1016.8\right]$ &  80.4 &  $\left[66, 96\right]$ &  19717.3 &  $\left[18745.8, 20257.2\right]$ \\ \hline
	\end{tabular}
	
	\caption{\label{tab_simuC3} \it Performance of $(\widehat{\lambda}_n,\widehat{\beta}_n)$ for different values of $\left(\lambda_n,\beta_n\right)$ using Algorithms (I) and (II), with continuous part defined by \eqref{processsimu} and jump size distribution  $0.4\left(-\mathcal{E}\left(15\right)\right) + 0.6\mathcal{E}\left(10\right)$. The threshold $v_n$ is chosen equal to $3\hat{\sigma}\Delta_n^{-0.01}$ with $\hat{\sigma}$ is the multi-power variation estimator of order 20. The means and quantile intervals 5\%-95\% are computed with $10^4$ Monte-Carlo simulations and $\Delta_n = 10^{-4}$.} 
\end{table}

\begin{table}[h!]
	\centering
	\begin{tabular}{|c|c|c|c|c|c|c|c|c|}
		\hline
		\multirow{3}{*}{$\left(\lambda_n, \beta_n\right)$}	&
		\multicolumn{4}{c|}{Algorithm (I)} & \multicolumn{4}{c|}{Algorithm (II)} \\
		\cline{2-9}
		 &\multicolumn{2}{c|}{$\widehat{\lambda}_n$} &\multicolumn{2}{c|}{$\widehat{\beta}_n$} &
		\multicolumn{2}{c|}{$\widehat{\lambda}_n$} &\multicolumn{2}{c|}{$\widehat{\beta}_n$} \\
		\cline{2-9}
		& Mean & C. I.   & Mean & C. I. 
		& Mean & C. I.   & Mean & C. I.  \\
		\hline
     $(10, 2)$ &   10.6 &     $\left[6, 16\right]$ &     2.0 &        $\left[-9.7, 13.7\right]$ &   5.8 &   $\left[2, 10\right]$ &     29.7 &         $\left[4.6, 46.4\right]$ \\
    $(10, 20)$ &   10.2 &     $\left[5, 16\right]$ &    18.4 &         $\left[2.4, 32.4\right]$ &   7.6 &   $\left[4, 12\right]$ &     29.0 &        $\left[16.8, 41.3\right]$ \\
   $(10, 200)$ &   54.8 &    $\left[39, 69\right]$ &  -256.6 &     $\left[-996.6, 114.1\right]$ &   9.3 &   $\left[5, 14\right]$ &    203.8 &      $\left[186.5, 224.9\right]$ \\ \hline
  $(10, 2000)$ &   81.5 &   $\left[47, 118\right]$ & -4216.3 &  $\left[-5908.7, -1902.0\right]$ &   9.6 &   $\left[5, 15\right]$ &   2002.8 &    $\left[1978.9, 2025.7\right]$ \\
 $(10, 20000)$ &   30.9 &    $\left[16, 48\right]$ & -1151.6 &   $\left[-1528.6, -742.1\right]$ &  10.8 &   $\left[5, 17\right]$ &  19778.4 &  $\left[19185.2, 20168.1\right]$ \\
     $(75, 2)$ &   74.2 &    $\left[61, 89\right]$ &     2.0 &         $\left[-2.4, 6.2\right]$ &  40.2 &  $\left[30, 51\right]$ &     40.5 &       $\left[-3.8, 221.3\right]$ \\
  \hline
    $(75, 20)$ &   72.3 &    $\left[59, 87\right]$ &    19.8 &        $\left[12.8, 27.8\right]$ &  49.2 &  $\left[39, 60\right]$ &     50.4 &       $\left[24.4, 158.9\right]$ \\
   $(75, 200)$ &   77.3 &    $\left[61, 97\right]$ &   178.8 &      $\left[139.1, 212.2\right]$ &  62.1 &  $\left[51, 74\right]$ &    225.4 &      $\left[198.2, 287.8\right]$ \\ 
  $(75, 2000)$ &  281.7 &  $\left[231, 330\right]$ & -1888.3 &  $\left[-2843.5, -1004.8\right]$ &  67.4 &  $\left[55, 81\right]$ &   2020.8 &    $\left[1959.3, 2114.6\right]$ \\
  \hline
 $(75, 20000)$ &  216.7 &  $\left[177, 257\right]$ & -1200.1 &  $\left[-1307.8, -1085.3\right]$ &  76.6 &  $\left[62, 91\right]$ &  19760.4 &  $\left[18778.1, 20302.9\right]$ \\
		\hline
	\end{tabular}
	\caption{\label{tab_simuC4} \it Performance of $(\widehat{\lambda}_n,\widehat{\beta}_n)$ for different values of $\left(\lambda_n,\beta_n\right)$ using Algorithmes (I) and (II), with continuous part defined by \eqref{processsimu} and jump size distribution  $0.4\left(-\mathcal{E}\left(15\right)\right) + 0.6\mathcal{E}\left(10\right)$. The threshold $v_n$ is chosen equal to $4\hat{\sigma}\Delta_n^{-0.01}$ with $\hat{\sigma}$ is the multi-power variation estimator of order 20. The means and quantile intervals 5\%-95\% are computed with $10^4$ Monte-Carlo simulations and $\Delta_n = 10^{-4}$.} 
\end{table}

\begin{table}[h!]
	\centering
	\begin{tabular}{|c|c|c|c|c|c|c|c|c|}
		\hline
		\multirow{3}{*}{$\left(\lambda_n, \beta_n\right)$}	&
		\multicolumn{4}{c|}{Algorithm (I)} & \multicolumn{4}{c|}{Algorithm (II)} \\
		\cline{2-9}
		 &\multicolumn{2}{c|}{$\widehat{\lambda}_n$} &\multicolumn{2}{c|}{$\widehat{\beta}_n$} &
		\multicolumn{2}{c|}{$\widehat{\lambda}_n$} &\multicolumn{2}{c|}{$\widehat{\beta}_n$} \\
		\cline{2-9}
		& Mean & C. I.   & Mean & C. I. 
		& Mean & C. I.   & Mean & C. I.  \\
		\hline
		$(10,2)$ & $9.95$  & $\left[5,15\right]$ & 2.01 & $\left[-9.21,12.85\right]$ &	 $5.45$  & $\left[2,9\right]$ & 19.5 & $\left[4.26,39.4\right]$ 	\\
		$(10,20)$ & $9.9$  & $\left[5,15\right]$ & 19.84 & $\left[7.6,32.4\right]$ &	 $7.5$  & $\left[3,12\right]$ & 28.8 & $\left[16.9,40.2\right]$ 	\\
		$(10,200)$ &  $33.4$  & $\left[17,47\right]$ & -76.7 & $\left[-540,158\right]$ 	& $9.3$  & $\left[5,14\right]$ & 204 & $\left[188,224\right]$ 	\\
		\hline
		$(10,2000)$ & $72.5$  & $\left[42,105\right]$ & -3968 & $\left[-5934,-1633\right]$ &	 $9.6$  & $\left[5,15\right]$ & 2002 & $\left[1979,2023\right]$ 	\\
		$(10,20000)$ & $29.2$  & $\left[15,46\right]$ & -1207 & $\left[-1508,-871\right]$ 	& $10.2$  & $\left[5,16\right]$ & 19861 & $\left[19337,20207\right]$ 	\\
		$(75,2)$ & $73.7$  & $\left[60,88\right]$ & 2 & $\left[-2.46,6.17\right]$ & $39.9$  & $\left[30,51\right]$ & 39.2 & $\left[-4,204\right]$ 	\\
		\hline
		$(75,20)$ & $71.9$  & $\left[59,86\right]$ & 19.8 & $\left[12.9,27.6\right]$ & $49.9$  & $\left[39,60\right]$ & 49.9 & $\left[24.5,155.6\right]$ 	\\
		$(75,200)$ & $68.7$  & $\left[55,83\right]$ & 191 & $\left[157,219\right]$ & $61$  & $\left[50,73\right]$ & 225 & $\left[198,291\right]$ 	\\
		$(75,2000)$ & $234$  & $\left[192,273\right]$ & -1403 & $\left[-2334,-556\right]$ & $65.6$  & $\left[54,78\right]$ & 2019 & $\left[1958,2109\right]$ 	\\
		\hline
		$(75,20000)$ & $207$  & $\left[170,247\right]$ & -1216 & $\left[-1308,-1113\right]$  & $74.2$  & $\left[60,89\right]$ & 19785 & $\left[18790,20310\right]$ 	\\
		\hline
	\end{tabular}
	\caption{\label{tab_simuC5} \it Performance of $(\widehat{\lambda}_n,\widehat{\beta}_n)$ for different values of $\left(\lambda_n,\beta_n\right)$ using Algorithmes (I) and (II), with continuous part defined by \eqref{processsimu} and jump size distribution  $0.4\left(-\mathcal{E}\left(15\right)\right) + 0.6\mathcal{E}\left(10\right)$. The threshold $v_n$ is chosen equal to $5\hat{\sigma}\Delta_n^{-0.01}$ with $\hat{\sigma}$ is the multi-power variation estimator of order 20. The means and quantile intervals 5\%-95\% are computed with $10^4$ Monte-Carlo simulations and $\Delta_n = 10^{-4}$.} 
\end{table}

\begin{figure}[h!]
    \centering
    \begin{subfigure}[b]{0.4\textwidth}
        \centering
       \includegraphics[width=\textwidth]{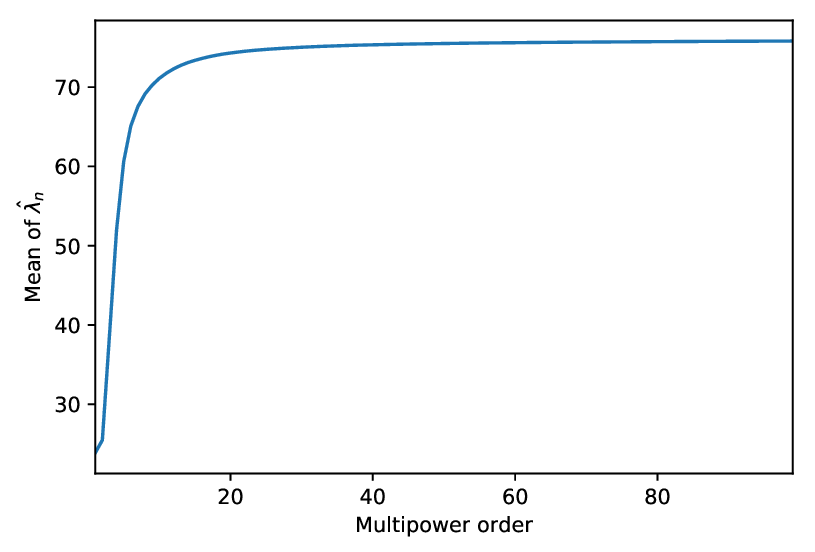}
        \caption{\it Estimation of $\lambda_n$.}
    \end{subfigure}
    \begin{subfigure}[b]{0.4\textwidth}
        \centering
       \includegraphics[width=\textwidth]{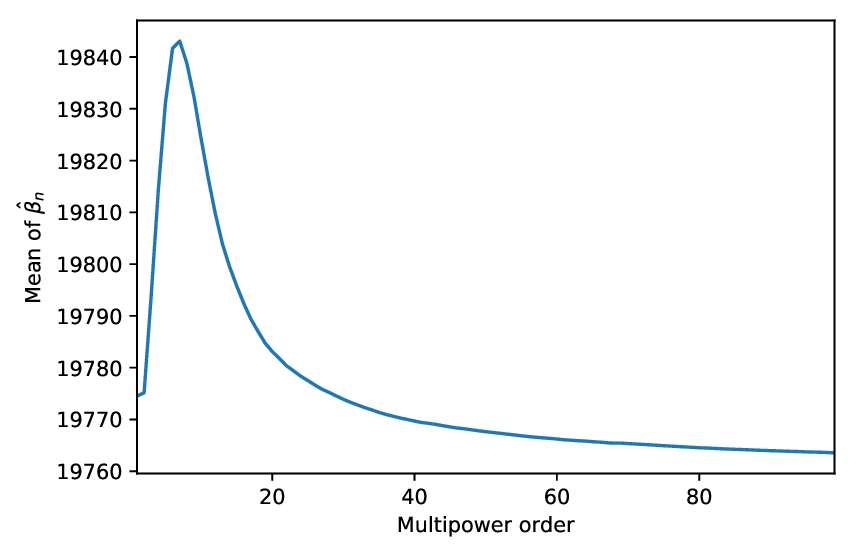}
        \caption{\it Estimation of $\beta_n$.}
    \end{subfigure}
      
     \caption{\it \label{impactmultipower} Estimation of $(\widehat{\lambda}_n,\widehat{\beta}_n)$ for different values of the order of the multipower estimator of the volatility $\hat{\sigma}$ using Algorithm (II), with continuous part defined by \eqref{processsimu}, $\lambda_n = 75$, $\beta_n = 20000$ and jump size distribution  $0.4\left(-\mathcal{E}\left(15\right)\right) + 0.6\mathcal{E}\left(10\right)$. The threshold $v_n$ is chosen equal to $5\hat{\sigma}\Delta_n^{-0.01}$. The mean of the estimators is computed with $10^4$ Monte-Carlo simulations and $\Delta_n = 10^{-4}$.}
\end{figure}

\subsection{Practical implementation on real data} \label{practical on real data}
Electricity spot historical data do exhibit spikes with strong mean reversion, see Figure \ref{spotdata}. We expect to obtain relatively high values for $\beta_n$, a necessary condition in order to apply our estimation procedure, especially under Assumption \ref{assumpestimatorn}\,(II). The goal is then to estimate the parameters $\lambda_n$ and $\beta_n$ of the process $Z^{\beta_n}$ on a time series of spot prices, assuming that the spot price is the sum of a continuous semimartingale and a spike process. We dispose of the following data: 
\begin{enumerate}
\item French electricity EPEX spot prices between the first of January of 2015 (included) and the first of January 2017 (not included) with data each hour \footnote{\label{epex}Source: \href{https://www.epexspot.com/}{https://www.epexspot.com/}},
\item German electricity EPEX spot prices between the first of January of 2015 (included) and the first of January 2017 (not included) with data each hour \footref{epex},
\item Australian electricity spot prices in Queensland between the first of January of 2015 (included) and the first of January 2017 (not included) with data each 30 minutes \footnote{Source: \href{https://www.aemo.com.au/}{https://www.aemo.com.au/}}.
\end{enumerate}
We estimate those parameters using a threshold $v_n = C \hat{\sigma} \Delta_n^{-0.01}$, with $\hat{\sigma}$ the multi-power variation of order 20 and $C$ a constant set to 3, 4 or 5. Results are presented in Table \ref{dataresults}. Equivalent half-life time in $hours^{-1}$, $\frac{\log(2)}{\hat{\beta}_n \Delta_n}$, and intensity $\frac{\hat{\lambda}_n}{2}$ in years$^{-1}$ are given Table \ref{dataresultsequivalent}. Figure \ref{spotdata} gives the time series of these three sets of data with jumps time estimated in the case $C = 5$. \\
As expected, the estimated $\widehat{\lambda}_n$ is sensitive to the value of $C$, the number of detected jumps decreasing with $C$. The estimated $\widehat{\beta}_n$ is much less sensitive, although we can observe a slight increase of values with $C$ in the French and German markets. We will see in section \ref{sec_call_option} the sensitivity of these estimators to the value of a strip of Call options.

%

 \begin{table}[h]
 	\centering
 	\begin{tabular}{|c|c|c|c|}
 		\hline
 		Market  & $C = 3$ & $C = 4$ & $C = 5$ \\
 		\hline
 		French  & $\left(100,19170 \right)$   & $\left(51,20259\right)$  &   $\left(35,21043\right)$ \\
 		German  &  $\left(145, 9848\right)$   & $\left(62,13438\right)$  &   $\left(34, 14531\right)$ \\
 		Australian  &    $\left(337, 22897\right)$   & $\left(227,22883\right)$  &   $\left(177,22884\right)$ \\
 		\hline
 	\end{tabular}
 	\caption{\label{dataresults}\it Estimation of $\left(\lambda_n, \beta_n\right)$ for different markets using a threshold of the form $v_n=C \hat{\sigma}\Delta_n^{-0.01}$ where $\hat{\sigma}$ is the multi-power variation estimator of order 20 and $C$ takes different values.} 
 \end{table}

 \begin{table}[h]
 	\centering
 	\begin{tabular}{|c|c|c|c|}
 		\hline
 		Market  & $C = 3$ & $C = 4$ & $C = 5$ \\
 		\hline
 		French  & $\left(50,0.63 \right)$   & $\left(25.5,0.60\right)$  &   $\left(17.5,0.58\right)$ \\
 		German  &  $\left(72.5, 1.23\right)$   & $\left(31,0.90\right)$  &   $\left(17, 0.83\right)$ \\
 		Australian  &    $\left(168.5, 1.06\right)$   & $\left(113.5,1.06\right)$  &   $\left(88.5,1.06\right)$ \\
 		\hline
 	\end{tabular}
 	\caption{\label{dataresultsequivalent}\it Estimation of intensity $\frac{\hat{\lambda}_n}{2}$ in years$^{-1}$ (first component)  and half-life time in $hours^{-1}$, $\frac{\log(2)}{\hat{\beta}_n \Delta_n}$ (second component), for different markets using a threshold of the form $v_n=C \hat{\sigma}\Delta_n^{-0.01}$ where $\hat{\sigma}$ is the multi-power variation estimator of order 20 and $C$ takes different values.} 
 \end{table}

\bigskip

In the previous analysis, one neglects the intraday seasonality which can be very strong and induce spikes detection, see Figure \ref{fig:intraday}. Indeed, the spot price is not an hourly process but a daily process. Spot fixing is done every day and the 24 prices for next day are settled simultaneously. However, representing this 24-dimensional stochastic process is very difficult and more importantly not suitable for risk management. Table \ref{tab:dailyestimation} contains the estimation results on the different data set with data sampled daily for each hour and for the daily average spot price using $C = 4$ for the threshold parameter (intensity $\frac{\hat{\lambda}_n}{2}$ in years$^{-1}$ and half-life time in days$^{-1}$, $\frac{\log(2)}{\hat{\beta}_n \Delta_n}$, are given). Our estimation procedure provides an estimation for jump times, jump sizes, and for the parameter $\beta$. However, the number of identified spike being low, the estimation procedure can not provide a robust estimation of $\lambda$. Figure \ref{fig:estimationdaily} represents the spike estimation for hours 0 and 18. Some spikes, for instance during January 2016 and hour 21, are not detected: they are not a spike in our modelling because for the spot to be high, several time steps are needed and it is not seen as a discontinuity.

\begin{figure}[h!]
    \centering
        \includegraphics[width=0.5\textwidth]{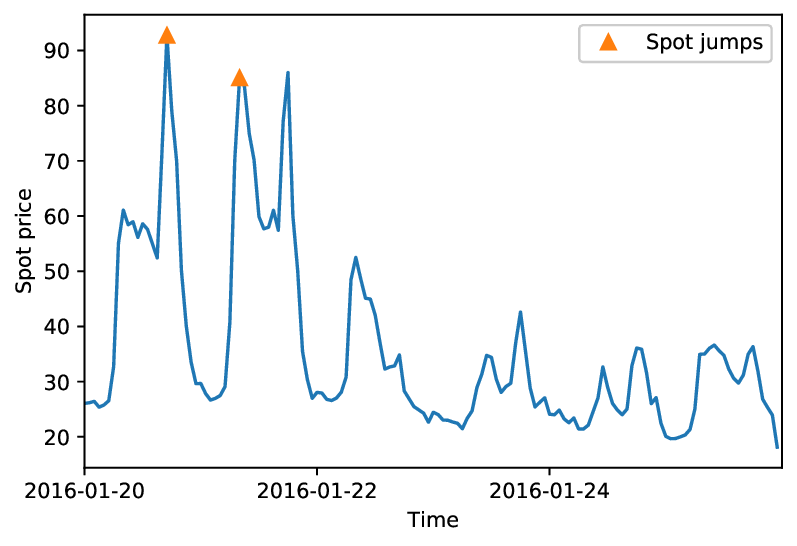}

    \caption{\it \label{fig:intraday} German hourly spot price between 2016-01-20 and 2016-01-25.}
\end{figure}

\begin{table}
\centering
\begin{tabular}{lrlrlrr}
\toprule
Country & \multicolumn{2}{l}{French} & \multicolumn{2}{l}{German} & \multicolumn{2}{l}{Australian} \\
Hour / Parameters & $\frac{\hat{\lambda}_n}{2}$ & $\frac{\log(2)}{\hat{\beta}_n \Delta_n}$ & $\frac{\hat{\lambda}_n}{2}$ & $\frac{\log(2)}{\hat{\beta}_n \Delta_n}$ & $\frac{\hat{\lambda}_n}{2}$ & $\frac{\log(2)}{\hat{\beta}_n \Delta_n}$ \\
\midrule
Day &               1.0 &           2.07 &               0.5 &                &              15.5 &           0.46 \\
\hline
0   &               1.0 &           1.35 &               8.5 &           1.23 &               5.0 &           1.76 \\
1   &               1.5 &           1.67 &               9.5 &           1.23 &               7.5 &           0.55 \\
2   &               0.5 &           2.67 &               8.5 &           0.65 &               4.5 &           0.82 \\
3   &               0.5 &           1.94 &               8.5 &           1.02 &               4.0 &           0.55 \\
\hline
4   &               1.0 &           3.73 &               5.5 &           0.78 &               4.5 &           0.46 \\
5   &               1.5 &           4.39 &               5.0 &           1.06 &               3.0 &           3.65 \\
6   &               0.0 &                &               2.5 &            0.8 &              11.0 &           0.85 \\
\hline
7   &               1.0 &                &               1.5 &           0.85 &              16.0 &           0.86 \\
8   &               1.5 &           0.72 &               0.5 &                &               7.5 &           0.59 \\
9   &               0.0 &                &               0.0 &                &              10.5 &           0.94 \\
\hline
10  &               1.0 &            2.1 &               0.0 &                &              13.0 &           0.87 \\
11  &               0.5 &            4.2 &               0.0 &                &               9.5 &           0.59 \\
12  &               1.5 &           4.02 &               2.0 &           0.84 &              14.0 &           1.13 \\
\hline
13  &               2.0 &           1.15 &               3.0 &            0.8 &              12.5 &           0.62 \\
14  &               1.5 &           1.87 &               4.0 &           0.46 &              14.0 &           0.54 \\
15  &               1.5 &           2.29 &               2.5 &           0.43 &              12.5 &           0.86 \\
\hline
16  &               1.0 &           2.24 &               2.0 &           0.44 &               9.0 &           1.02 \\
17  &               1.5 &           1.85 &               2.5 &           4.46 &              10.5 &           1.55 \\
18  &               3.5 &           1.27 &               2.5 &           1.85 &              12.5 &           3.04 \\
\hline
19  &               3.5 &           1.36 &               1.0 &           7.14 &              11.0 &           1.14 \\
20  &               1.5 &           1.39 &               1.0 &           0.65 &              10.5 &           2.88 \\
21  &               1.5 &           5.08 &               1.5 &           0.55 &               7.0 &           0.30 \\
\hline
22  &               2.0 &           0.96 &               2.0 &           0.71 &               5.0 &           1.05 \\
23  &               1.0 &           0.52 &               4.0 &           0.87 &              15.0 &           3.11 \\
\bottomrule
\end{tabular}
\caption{\it \label{tab:dailyestimation} Estimation of $\left(\lambda_n, \beta_n\right)$ on daily data for the different hours of the day (hour 0 to 23) and for daily average price (Day) using a threshold of the form $v_n=4 \hat{\sigma}\Delta_n^{-0.01}$ where $\hat{\sigma}$ is the multi-power variation estimator of order 20. Intensity $\frac{\hat{\lambda}_n}{2}$ in years$^{-1}$ and half-life time in $days^{-1}$, $\frac{\log(2)}{\hat{\beta}_n \Delta_n}$, are given. Half-life time values are missing when there are no jumps or when the left hand part in the maximum term of \eqref{construct estim beta} is less than 0.}
\end{table}

\begin{figure}[h!]
    \centering
    \begin{subfigure}[b]{0.49\textwidth}
        \centering
        \includegraphics[width=\textwidth]{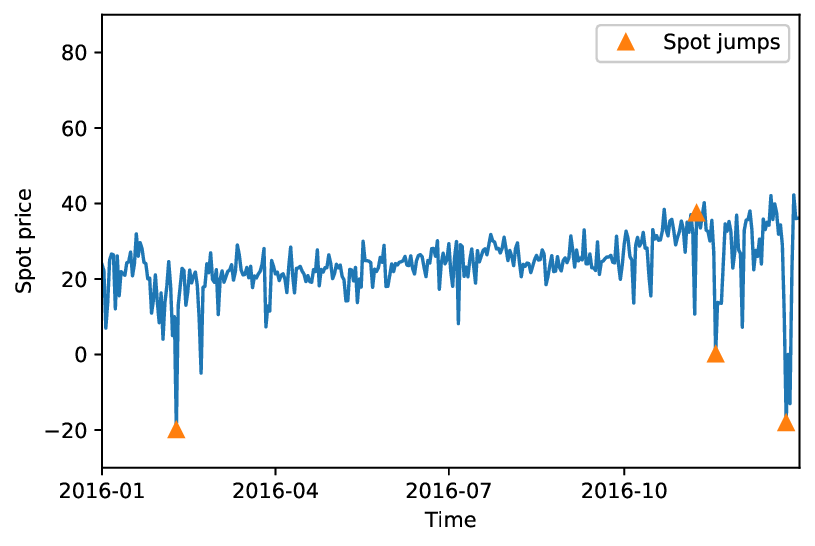}
        \caption{\it Hour 0.}
    \end{subfigure}
    \begin{subfigure}[b]{0.49\textwidth}
        \centering
        \includegraphics[width=\textwidth]{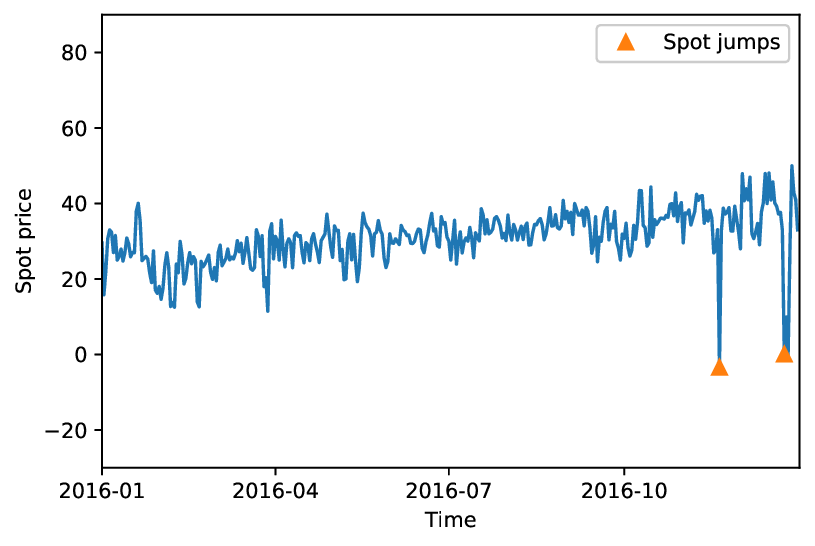}
        \caption{\it Hour 21.}
    \end{subfigure}
    \caption{\it \label{fig:estimationdaily} Spike detection on daily German data during year 2016 for hour 0 and hour 21 using a threshold of the form $v_n=4 \hat{\sigma}\Delta_n^{-0.01}$ where $\hat{\sigma}$ is the multi-power variation estimator of order 20.}
\end{figure}

\section{Derivative pricing in the electricity market}
\label{application}
\subsection{Forward price formula}
Due to the non-storability of electricity, the spot price is not an asset. The modelling of (and then an analytical formula for) the forward prices (i.e. the real assets and hedging products) is essential for risk management purposes. 
The question of the choice of the risk-neutral probability is addressed in the next section. Here, we consider that the electricity spot price $(S_t)_{t \geq 0}$ (respectively the logarithm of the spot price $(S_t)_{t \geq 0}$)  is modelled by $S_t = X_t^c + Z_t^{\beta}$ (respectively $\log(S_t) = X_t^c + Z_t^{\beta}$), according to \eqref{our model first part}-\eqref{our model second part}-\eqref{our model third part} under a risk-neutral probability $\mathbb{Q}$.  \\

The forward price $f(t,T)$ quoted at time $t$, delivering 1MWh at time $T$, is defined by:
\[f\left(t,T\right) = \mathbb{E}^{\mathbb{Q}}\big[S_T \,|\, \mathcal{F}_t \big].\]
The available contracts in the electricity markets are of the form $f(t,T,\theta)$: a contract that delivers 1MWh continuously from $T$ to $T+\theta$. The delivery period $\theta$ can be one week, one month, one year and so on. For example, the contract called "one-week-ahead" (1WAH) will deliver 1MWh continuously from the first hour of next Monday to the last hour of the following Sunday; the contract called "one-month-ahead" (1MAH) will deliver 1MWh continuously between the first and the last hour of next month. By classical no arbitrage arguments \cite{benth07} the price of such a product is defined by: 
\[f\left(t,T,\theta\right) = \frac{1}{\theta} \int_{T}^{T+\theta} f\left(t,u\right) du.\]
Theorems \ref{forward} and \ref{forwardlog} give analytic formulae for the forward price $f\left(t,T\right)$ respectively for the modelling of the spot price and the modelling of the logarithm of the spot price by \eqref{our model first part}-\eqref{our model second part}-\eqref{our model third part}. 
\begin{theorem} \label{forward} Suppose that the spot price is modelled by
$S_t = X_t^c + Z_t^{\beta},\;\;t\geq 0$,
according to \eqref{our model first part}-\eqref{our model second part}-\eqref{our model third part}
under a risk-neutral probability $\mathbb{Q}$. 
\begin{enumerate}
\item We have an explicit representation of
$f\left(t,T\right) = \mathbb{E}^{\mathbb{Q}}\big[ S_T\,\big|\, \mathcal{F}_t \big]$ given by 
\[f\left(t,T\right) = f^c\left(t,T\right) + f^{\beta}\left(t,T\right),\]
with 
\[f^c\left(t,T\right) = \mathbb{E}^{\mathbb{Q}}\big[X^c_T \,\big|\, \mathcal{F}_t \big]\]
and
\[f^{\beta}\left(t,T\right) = e^{-\beta \left(T-t\right)}Z_t^{\beta} + \frac{\lambda  x \star \nu }{\beta}\big(1-e^{-\beta\left(T-t\right)}\big).\]
\item We also have  $f\left(t,T,\theta\right) = f^c\left(t,T,\theta\right) + f^{\beta}\left(t,T,\theta\right)$, with
$f^c\left(t,T,\theta\right) = \frac{1}{\theta} \int_{T}^{T+\theta} f^c\left(t,u\right) du$
and 
\[f^{\beta}\left(t,T,\theta\right) = e^{-\beta(T-t)}\left(\frac{1-e^{-\beta \theta}}{\beta \theta} \right) Z_t^{\beta} + \frac{\lambda  x \star \nu }{\beta}\left[1-e^{-\beta(T-t)}\left(\frac{1-e^{-\beta \theta}}{\beta \theta} \right)\right].\]

\end{enumerate}
\end{theorem}

The proof is straightforward. 
Theorem \ref{forward} has three major consequences.
\begin{enumerate}
\item The proposed model \eqref{our model first part}-\eqref{our model second part}-\eqref{our model third part} allows us to get analytical formulas for the forward prices $f(t,T)$ and $f(t,T,\theta)$, provided that we have analytical formulas for the continuous part $f^c\left(t,T\right)$. This then covers a wide range of models. 
\item We can easily write the model \eqref{our model first part}-\eqref{our model second part}-\eqref{our model third part} in its equivalent form on forward prices by $f(t,T)=f^c\left(t,T\right) + f^{\beta}\left(t,T\right)$ with
\begin{equation}
\label{eq_fbeta}
f^{\beta}\left(t,T\right) = \int_{0}^t  \int_{\mathbb{R}} x e^{-\beta\left(T-t\right)} \underline{p}\left(ds,dx\right). 
\end{equation}
It is then easy to consider the proposed model as an extension of any classical (continuous) model written on the forward prices, allowing to represent spikes in the spot price dynamics. 
\item If $\lambda/\beta$ is small and $\beta$ is large, the impact of $f^\beta$ on the forward prices is negligible and the additive spike process has only an impact on the spot prices. This is consistent with the observations in the electricity markets, the forward prices showing no spikes. 
\end{enumerate}
These consequences are of significant importance. Especially in the case where $\lambda/\beta$ is small and $\beta$ is large, this means that the spike process can be treated independently, both in parameter estimation and in simulation. Indeed, consider any existing (continuous) model describing (or simulating) $f^c(t,T)$ and calibrated on forward prices, the proposed model then consists in adding (simulations of) the spike process calibrated on spot prices following the estimation procedure previously described.  
\begin{theorem} \label{forwardlog} Suppose that the logarithm of the spot price is modelled by
$\log(S_t) = X_t^c + Z_t^{\beta},\;\;t\geq 0$,
according to \eqref{our model first part}-\eqref{our model second part}-\eqref{our model third part}
under a risk-neutral probability $\mathbb{Q}$. Let assume that $\int_{\mathbb{R}} e^{ux} \nu\left(dx\right) < \infty$ for all $u \in \left[0,1\right]$.
We have an explicit representation of
$f\left(t,T\right) = \mathbb{E}^{\mathbb{Q}}\big[ S_T\,\big|\, \mathcal{F}_t \big]$ given by 
\[f\left(t,T\right) = f^c\left(t,T\right)f^{\beta}\left(t,T\right),\]
with 
\[f^c\left(t,T\right) = \mathbb{E}^{\mathbb{Q}}\big[e^{X^c_T} \,\big|\, \mathcal{F}_t \big]\]
and
\[f^{\beta}\left(t,T\right) = e^{e^{-\beta \left(T-t\right)}Z_t^{\beta}}e^{ \frac{\lambda}{\beta}  \int_0^1 (\int_{\mathbb{R}}e^{ux}\nu\left(dx\right)-\int_{\mathbb{R}}e^{ue^{-\beta(T-t)}x}\nu(dx))du }.\]
\end{theorem}

The proof is straightforward and comments for the results of Theorem \ref{forward} can be transposed to the ones of Theorem \ref{forwardlog}. In particular, if $\lambda/\beta$ is small and $\beta$ is large, the term $f^{\beta}\left(t,T\right)$ is close to 1 and $f(t,T)$ to $f^c(t,T)$. A similar formula is given for the forward prices in \cite{hambly09} in the case where $X^c$ follows a Gaussian Ornstein-Uhlenbeck process.

\subsection{Specific model and change of measure}
\label{modelappli}
We address the problem of choosing the risk neutral probability which we illustrate with a specific and simple model: in the rest of this section, we consider the  model defined by 
\[f\left(t,T\right) = \int_{0}^t \mu_s ds + f^c\left(t,T\right) + f^{\beta}\left(t,T\right),\]
with $f^\beta(t,T)$ defined by \eqref{eq_fbeta} and where the continuous part $f^c(t,T)$ follows the dynamics
\[
df^c\left(t,T\right) = f^c\left(t,T\right)(\sigma_l dW_t^l + \sigma_s e^{-\alpha\left(T-t\right)}dW_t^s )
\]
with $(W^l_t, W^s_t)_{t \geq 0}$ a two-dimensional Brownian motion under the historical probability $\mathbb{P}$ with correlation $\rho$. This dynamics corresponds to a classical two factors model for forward prices of electricity \cite{kiesel09} or gas \cite{warin2012gas}. The forward price is driven by a short term factor with volatility $ \sigma_s e^{-\alpha\left(T-t\right)}$ and a long term factor with volatility $\sigma_l$. 
The short term volatility $ \sigma_s e^{-\alpha\left(T-t\right)}$ captures the Samuelson effect: the volatility increases when $T-t$ decreases. 
The spot price is then equal to 
$S_t = \int_{0}^t \mu_s ds + X^c_t + Z_t^{\beta}$
with 
\[X_t^c = f^c\left(0,t\right)\exp\left(-\frac{1}{2}\left[\sigma_l^2 t + \sigma_s^2 \frac{1-e^{-2\alpha t}}{2 \alpha} + 2 \rho \sigma_l \sigma_s \frac{1-e^{-\alpha t}}{\alpha} \right]  + \sigma_l dW^l_t + \sigma_s \int_{0}^t e^{-\alpha\left(t-u\right)}dW_u^s\right),\]
and the model then falls within the class of models \eqref{our model first part}-\eqref{our model second part}-\eqref{our model third part}.
\\
 
We have seen that the forward products are not impacted by the spikes if $\lambda/\beta$ is small and $\beta$ is large. However, it can have an important impact on options on the spot, for instance on a strip of Call options, with payoff of the form 
$\sum_{i=1}^{I} \left(S_{t_{i,n}} - K\right)^+$
for prescribed dates $t_{i,n}$.
If we consider an option with payoff having a single component $\left(S_t - K\right)^+$, the jump process will have a weak impact: the probability to have a jump at time $t$ is equal to 0 and even if there is a jump before, it disappears very quickly. However, the jump process may have a significant impact on the value of options with payoff $\sum_{i=1}^{I} \left(S_{t_{i,n}} - K\right)^+$ because the probability of having spikes on $\left[0,1\right]$ is non zero. (Note that only upward spikes will have an impact on the price of these options.)

\medskip
Unlike spot prices, forward contracts are tradable assets. In the following, we assume absence of arbitrage opportunity. According to the fundamental theorem of asset pricing, there exists a probability measure $\mathbb{Q}$ equivalent to the historical measure $\mathbb{P}$ such that $f\left(t,T\right)$ is a local martingale under $\mathbb{Q}$\footnotemark[3]. Because of the presence of jumps, the market is incomplete and $\mathbb{Q}$ is not unique. According to \cite[Theorem 2]{protter08}, there exists a predictable process $\left(\gamma_t\right)_{t \geq 0}$ and a predictable process $\left(Y\left(t,x\right)\right)_{t\geq 0, x \in \mathbb{R}}$ such that: 
\begin{enumerate}
\item[1)] $\mu_t  + \gamma_t c_t + \int_{0}^t \int_{\mathbb{R}} x\lambda  Y\left(t,x\right) e^{-\beta\left(T-t\right)} \nu\left(dx\right) = 0$ ($\mathbb{P} \otimes dt$ almost-surely),
\item[2)] $\int_{0}^1 \gamma_s^2 c_s ds < \infty$ almost surely, 
\item[3)] $\int_{0}^1 \int_{\mathbb{R}} |x|^2 \wedge |x| Y\left(t,x\right) e^{-\beta\left(T-t\right)} \lambda \nu\left(dx\right) < \infty$ ($\mathbb{P} \otimes dt$ almost -surely).
\end{enumerate}
with $c_t$ equal to $f^c\left(t,T\right)\left(\sigma_l^2 + \sigma_s^2 e^{-2\alpha\left(T-t\right)}+ 2 \rho \sigma_l \sigma_s e^{-\alpha\left(T-t\right)} \right)^{1/2}$ in our case. Under the equivalent measure, $f\left(t,T\right)$ is an It\^o semi-martingale with drift 0, volatility $c_t$ and jump measure $p^* = Y p$ following 
\[df\left(t,T\right) = df^c\left(t,T\right) + df^{\beta}\left(t,T\right)\]
with 
\[
df^c\left(t,T\right) = f^c\left(t,T\right)\left(\sigma_l dW_t^{l,*} + \sigma_s e^{-\alpha\left(T-t\right)}dW_t^{s,*} \right),
\]
\[df^{\beta}\left(t,T\right) = \int_{\mathbb{R}} x e^{-\beta\left(T-t\right)} \left(\underline{p^*}\left(dt,dx\right) -\lambda Y\left(t,x\right) \nu\left(dx\right)dt\right)\]
for two Brownian motions $(W^{s,*},W^{l,*})$ under the new measure. The volatility does not change unlike the intensity and the law of jump sizes of the Poisson process.

\medskip
In order to choose the change of martingale measure, one usually considers an optimisation criterion. One of the most common used criterion is the local risk-minimisation introduced by F\"ollmer and Schweizer (see \cite{schweizer01} for details). The variance of the cost of the strategy is minimised locally, infinitesimally at each time. This strategy corresponds to choose the minimal martingale measure defined in \cite{follmer91}. Under certain assumptions, this measure is a true probability measure and the asset is a local martingale under this measure. Furthermore, the intensity changes and depends on the drift $\mu$, which is also true for most common criteria. Since we work on a finite time framework, the drift is not identifiable and it is not possible to estimate it. \\

In the following we choose the historical approach of Merton consisting in picking a change of probability that does not affect the intensity and the jump sizes of the Poisson measure \cite{merton76}. The equivalent probability measure is defined by
\[\frac{d\mathbb{Q}^M}{d\mathbb{P}} = \exp\big(\int_{0}^1 \theta_u d\tfrac{\sigma_l W^l_u + \sigma_s e^{-\beta\left(T-u\right)}W^s_u}{\left(\sigma_l^2 + \sigma_s^2 e^{-2\alpha\left(T-u\right)} + 2 \rho \sigma_l \sigma_s e^{-\alpha\left(T-u\right)} \right)^{1/2}} - \frac{1}{2}\int_{0}^1 \theta_u^2 du\big)\]
with 
$\theta_u = \frac{-(\mu_u + e^{-\beta(T-u)\int_{\mathbb{R}} x \nu(dx)})}{f^c(u,T)\left(\sigma_l^2 + \sigma_s^2 e^{-2\alpha\left(T-u\right)} + 2 \rho \sigma_l \sigma_s e^{-\alpha\left(T-u\right)} \right)^{1/2}}.$
The Novikov's condition is satisfied so it defines in turn a genuine probability measure. Under $\mathbb{Q}^M$, the price of the forward contract $f\left(t,T\right)$ follows the dynamics 
$df\left(t,T\right) = df^c\left(t,T\right) + df^{\beta}\left(t,T\right)$
with 
\[
df^c\left(t,T\right) = f^c\left(t,T\right)\big(\sigma_l dW_t^{l,\mathbb{Q}^M} + \sigma_s e^{-\alpha\left(T-t\right)}dW_t^{s,\mathbb{Q}^M} \big)
\]
and
\[df^{\beta}\left(t,T\right) = \int_{\mathbb{R}} x e^{-\beta\left(T-t\right)} \left(\underline{p}\left(dt,dx\right) -\lambda \nu\left(dx\right)dt\right),\]
where $W^{l,\mathbb{Q}^M}$ and $W^{s,\mathbb{Q}^M}$ are two $\mathbb{Q}^M$-Brownian motions. Merton chooses this probability considering that the risk associated to the jumps is diversifiable. As noticed in Tankov in \cite[Section 10.1]{tankov03}, using this strategy leaves one exposed to the risk of the jumps. It only corrects the average effect of jumps (provided that the jump component of the electricity price is independent of the other assets, which is the case here: we understand the electricity spikes caused by physical exogenous events;
it can in particular be caused by the production capacity and the demand which are not assets (see the structural model of Aid et al. \cite{aid09} for instance). 
Finally, the price of an option with payoff $H\left(S_T\right) = H\left(f\left(T,T\right)\right)$ at time t is given by 
$\mathbb{E}^{\mathbb{Q}^M}\big[H\left(S_T\right) \,|\, \mathcal{F}_t\big]$.


\subsection{Application to Call option pricing in the French market}
\label{sec_call_option}

In the following, we focus on the French market and we work on the model of Section \ref{modelappli}. We dispose of the hourly spot and daily forward prices in 2015 and 2016.

\medskip
\paragraph{{\bf Parameters of $f^{\beta}$}} We use the parameters found in Table \ref{dataresults} to calibrate $Z^{\beta}$ to the spot prices. We model the size of the jumps by its empirical distribution, each jump being estimated with $\Delta^n_{\mathcal{I}_n\left(q\right)} X$, knowing that a bias (mentioned in the end of section \ref{sec_estimation_jumps}) is present.

\medskip
\paragraph{{\bf Parameters of $f^c$}}
We consider the following forward products in the French market: 1 to 4 Week-ahead, 1 to 3 Month-ahead, 1 to 4 Quarter-ahead and 1 and 2 Year-ahead products. As $\widehat{\lambda}_n/\widehat{\beta}_n$ is small and $\widehat{\beta}_n$ is large, we can neglect the jump part on the forward prices and consider that the forward products have only a continuous part. We use the method of F\'eron and Daboussi \cite{feron15} to calibrate the parameters of $f^c$ to the observed forward prices. We find for the different parameters $\alpha = 12.56 \; y^{-1}$, $\sigma_s = 1.03 \; y^{-\frac{1}{2}}$, $\sigma_l = 0.25 \; y^{-\frac{1}{2}}$ and $\rho = -0.11$.

\medskip
\paragraph{{\bf Forward products}} In Figure \ref{simuproduct}, we display a simulation of the spot price, the 1WAH and the 1MAH with and in absence of spikes. The parameters of the spike component are the one of Table \ref{dataresults} with $C=5$. We observe that the difference between the trajectory of the forward products with and without spikes is very small but significant for the spot price.

\begin{figure}[h!]
    \centering
    \begin{subfigure}[b]{0.35\textwidth}
        \centering
        \includegraphics[width=\textwidth]{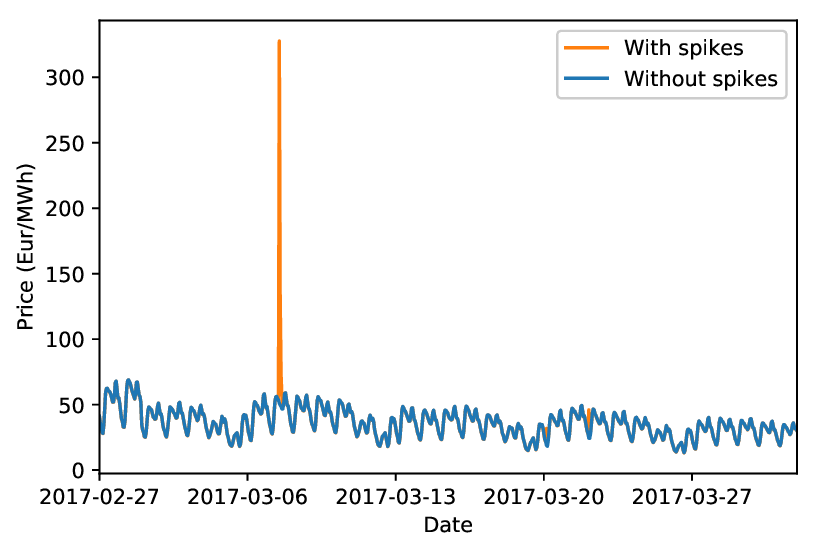}
        \caption{\it Spot.}
    \end{subfigure}
    \begin{subfigure}[b]{0.35\textwidth}
        \centering
        \includegraphics[width=\textwidth]{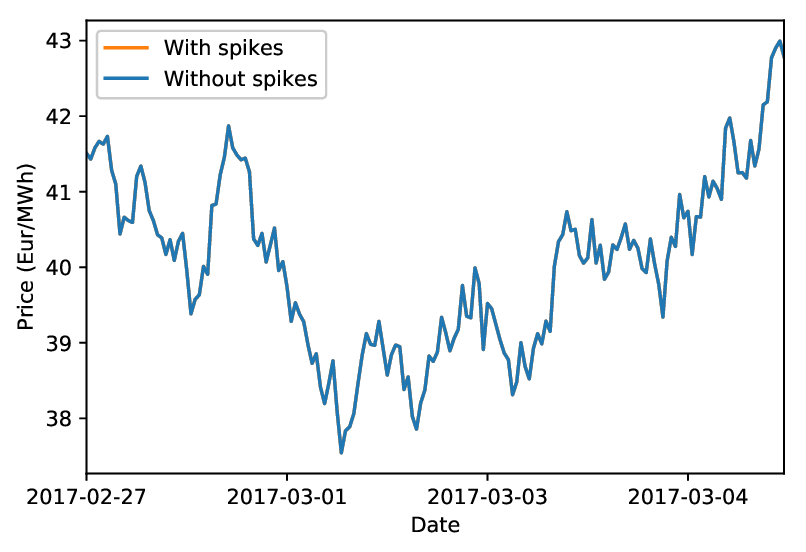}
        \caption{\it 1WAH.}
    \end{subfigure}
      
        \begin{subfigure}[b]{0.35\textwidth}
        \centering
        \includegraphics[width=\textwidth]{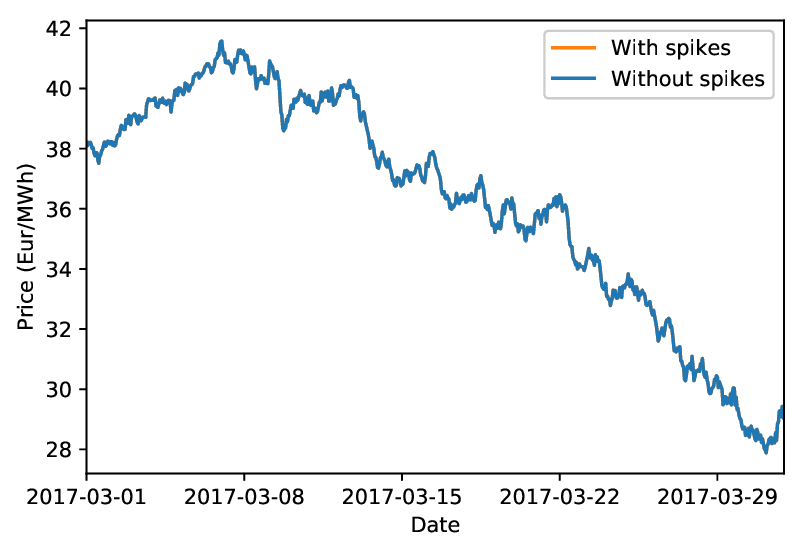}
        \caption{\it 1MAH.}
    \end{subfigure}
    \caption{\it \label{simuproduct} Simulation of different products in a two factor model with and without spikes between the $27^{th}$ of February 2017 and the $31^{st}$ of March 2017. We illustrate the spot, the 1WAH starting the $27^{th}$ of February 2017 and the 1MAH starting the $01^{st}$ of March 2017.}
\end{figure}

\medskip

\paragraph{{\bf Strip of Call options}}  We consider an option of payoff $\sum_{i=1}^{I} \left(S_{t_{i,n}} - K\right)^+$, with $I=8760$ corresponding to possible exercises each hour of one year. This choice of payoff is motivated by the valuation of a power plant: the produced electricity is sold on the market at $S$ and the production cost is $K$. A high strike corresponds to peaking power stations activated when other plants can not satisfy the total load. The price of such an option is equal to $\mathbb{E}^{\mathbb{Q}^M}\big[\sum_{i=1}^{I} \left(S_{t_{i,n}} - K\right)^+\big]$\footnotemark[3]. We give in Table \ref{pricingresults} confidence intervals at level $95\%$ for the option price with the different strikes $100$, $200$ and $300$, computed using Monte Carlo method with $10000$ simulations. We consider the case where there is no spikes and the cases with spikes using the different threshold of the form $v_n = C \hat{\sigma} \Delta_n^{-0.01}$ with $C = 3$, $C = 4$ and $C = 5$. Considering spikes leads to higher value for the strip options. Furthermore, options valued at zero have now non negligible values. We notice that the choice of the threshold have an impact on the price of the option. Indeed, a higher $C$ leads to less jumps and a smaller $\lambda$, see Table \ref{dataresults}, and then to a lower price as most of the jumps are positive. However, the impact is low since we keep the larger jumps when $C$ is increased  which are the ones impacting the price for high strikes.

\footnotetext[3]{For simplicity, we consider a risk-free rate equal to zero}
 \begin{table}[h]
 \centering
\begin{tabular}{|c|c|c|c|}
\hline
Model/Strike  & 100 &  200 & 300 \\
\hline
Without spike  & $\left[1716.22 , 1806.83\right]$   & $\left[0.0089 , 0.063 \right]$&  $\left[0,0\right]$ \\
Spikes, $C = 3$  &  $\left[2482.57 , 2576.17 \right]$   & $\left[434.21 , 450.26 \right]$&  $\left[264.14, 274.00\right]$ \\
Spikes, $C = 4$  &    $\left[2442.66 , 2536.26\right]$   & $\left[412.98 , 428.44\right]$&  $\left[251.72 , 262.29\right]$ \\
Spikes, $C = 5$  &   $\left[2417.24 , 2510.79 \right]$   & $\left[397.74 , 412.76 \right]$&  $\left[242.04 , 252.41\right]$ 
\\
\hline
\end{tabular}
 \caption{\label{pricingresults} \it Confidence intervals at level $95\%$ for the price of strip options computed using Monte Carlo method with $10000$ simulations for different strikes and different models.} 
 \end{table}

 
 \subsection{A word of conclusion}
In conclusion, we see that our statistical methodology and estimators are feasible on simulated data, with quite satisfactory results. For practical implementation, both for the statistical estimation of the parameters and for option pricing, the picture is fairly good in most cases. Depending on the market structure and location, we meet some regimes that are compatible with our assumptions and our asymptotic regimes. However, we fail to capture in a robust way the versatility of every market. This is explained by two facts: firstly, our jump model is too simple to apply in all situations; secondly, we rely on a high number of spikes and fast mean-reversion, and this is not always the case in practice. One could certainly devise fine (and non-asymptotic) statistical methods to incorporate this issue in some cases, at the risk nonetheless of loosing the robustness with respect to the underlying continuous part of price model, which is completely nonparametric in the present case.

\section{Proofs}
\label{proofs}

In the following proofs, we set the drift $(\mu_t)_{t \geq 0}$ vanishes identically. Generalizing to the non-null drift case is done using the usual argument based on Girsanov theorem.  Also, we assume for simplicity that $(\sigma_t)_{t \geq 0}$ is a deterministic function, in order to simplify in particular the proofs for central limit theorems.

\subsection{Proof of Proposition \ref{estimatorn}}

The proof follows the path of \cite[Theorem 10.26, p.374]{ait14} and is also close to Mancini \cite{mancini04} in spirit. 
We will denote by $\xi_\nu$ a generic random variable distributed according to $\nu$. 
Let 
$$\mathcal {\mathcal A}_n = \big\{i \in \{1,\ldots, n\}, i \neq i\left(n,q\right)\;\forall q \geq 1\big\}$$
be the set of indices $i$ such that the interval $((i-1)\Delta_n, i\Delta_n]$ contains no jump. 

\subsubsection*{Proof of Proposition \ref{estimatorn} under Assumption \ref{assumpestimatorn} (I)}
We first need to show 
\begin{equation} \label{convsup}
\mathbb{P}\big(\max_{i \in {\mathcal A}_n} \tfrac{|\Delta^n_i X|}{\sqrt{\Delta_n}} > v_n  \big) \rightarrow 0,
\end{equation}
\begin{equation} \label{convsup2}
\mathbb{P}\big(\min_{i \in {\mathcal A}_n^c} \tfrac{|\Delta^n_i X|}{\sqrt{\Delta_n}} < v_n  \big) \rightarrow 0
\end{equation} and 
\begin{equation} \label{no2jumps}
\mathbb{P}\big(\max_{1 \leq i \leq n} \Delta^n_i N \geq 2\big) \rightarrow 0.
\end{equation}

We have
\[\frac{\Delta_i^n X}{\sqrt{\Delta_n}} = \frac{-\beta_n  \int_{t_{i-1,n}}^{t_{i,n}} Z^{\beta}_s ds}{\sqrt{\Delta_n}} + \frac{\Delta^n_i X'}{\sqrt{\Delta_n}},\]
with 
\[X'_t = \int_0^t \mu_s ds + \int_0^t \sigma_s dW_s + \int_0^t \int_{\mathbb{R}} x \underline{p}\left(ds,dx\right).\]
By \cite[Equation (10.71), p.374]{ait14} we have
$\mathbb{P}\big(\underset{i \in {\mathcal A}_n}{\sup} \frac{|\Delta^n_i X'|}{\sqrt{\Delta_n}} > v_n  \big) \rightarrow 0,$
Therefore, in order to prove \eqref{convsup}, we need to show that 
\begin{equation} \label{convsupz}
\mathbb{P}\big(\underset{i \in {\mathcal A}_n}{\sup}\tfrac{|\beta_n \int_{t_{i-1,n}}^{t_{i,n}} Z^{\beta}_{s} ds|}{\sqrt{\Delta_n}} > v_n\big)  \rightarrow 0.
\end{equation}
Since $|\beta_n \int_{t_{i-1,n}}^{t_{i,n}} Z^{\beta}_s ds| = \left(1-e^{-\beta_n \Delta_n}\right) |Z^{\beta}_{t_{i-1,n}}|$ for $i \in {\mathcal A}_n$, we have 
\begin{align*}
\PP\big(\max_{i \in {\mathcal A}_n}\frac{|\beta_n \int_{t_{i-1,n}}^{t_{i,n}} Z^{\beta}_{s} ds|}{\sqrt{\Delta_n}} > v_n\big)& \leq \PP\big(\beta_n \sqrt{\Delta_n}\sup_{t \in [0,1]}\big| Z_t^\beta \big| > v_n\big)\\
& \leq \PP\big(\sup_{t \in [0,1]} \int_0^t \int_{\mathbb{R}} |x| e^{-\beta_n\left(t -s\right)} \underline{p}\left(ds,dx\right) >  \frac{v_n}{\beta_n \sqrt{\Delta_n}} \big) \\
& \leq 2\lambda_n \mathbb{P}\big( |\xi_\nu| > \frac{v_n}{\beta_n \sqrt{\Delta_n}}\big)
\end{align*}
by Markov's inequality and  \cite[Equation (10)]{bierme12} on the expectation of the crossings of shot noise processes. This last term converges to $0$ by Assumption \ref{assumpestimatorn} (ii) and \eqref{convsupz} follows which completes the proof of \eqref{convsup}. \\


We next turn to \eqref{no2jumps}. The left hand side of \eqref{no2jumps} is equal to 
$\mathbb{P}\left(\cup_{i=1}^n \Delta^n_i N \geq 2 \right) \lesssim \lambda_n^2 \Delta_n$
which converges to 0 if $\lambda_n \sqrt{\Delta_n} \rightarrow 0$. With no loss of generality we may (and will) work on the set $\{\underset{1 \leq i \leq n}{\max} \Delta^n_i N \leq 1\}$.  In the interval $\left(\left(i\left(n,q\right)-1\right)\Delta_n,i\left(n,q\right)\Delta_n\right]$, there is only one jump 
and we have \[
\Delta_{i\left(n,q\right)}^n Z^{\beta} = -\left(1-e^{-\beta_n \Delta_n}\right) Z_{t_{i\left(n,q\right)-1,n}} + e^{-\beta_n\left(i\left(n,q\right)\Delta_n-T_q\right)} \Delta X_{T_{q}}\]
for all $q \geq 1$, therefore
\begin{align*}
\frac{|\Delta^n_{i\left(n,q\right)} X|}{\sqrt{\Delta_n}} &\geq \frac{e^{-\beta_n\Delta_n}|\Delta X_{T_{q}}|}{\sqrt{\Delta_n}} - \frac{|\Delta_{i\left(n,q\right)}^n X^c|}{\sqrt{\Delta_n}} - \frac{|\left(1-e^{-\beta_n \Delta_n}\right) Z^{\beta}_{t_{i\left(n,q\right)-1,n}}|}{\sqrt{\Delta_n}}\\
&\geq \underset{1 \leq q \leq N_1}{\min} \frac{e^{-\beta_n\Delta_n}|\Delta X_{T_{q}}|}{\sqrt{\Delta_n}} -  \underset{i \in {\mathcal A}_n^c}{\max} \frac{|\Delta_{i}^n X^c|}{\sqrt{\Delta_n}}  - \max_{1 \leq i \leq n} \frac{\beta_n \Delta_n |Z^{\beta}_{t_{i,n}}|}{\sqrt{\Delta_n}}.
\end{align*}
It follows that $\mathbb{P}\big(\underset{i \in {\mathcal A}_n^c}{\min} \frac{|\Delta^n_i X|}{\sqrt{\Delta_n}} \leq v_n  \big)$ is dominated by the sum of the three following terms:
\begin{equation} \label{dominf1}
\mathbb{P}\big( \underset{1 \leq q \leq N_1}{\min} |\Delta X_{T_{q}}|e^{-\beta_n\Delta_n} \leq 3 v_n \sqrt{\Delta_n}\big), 
\end{equation}
\begin{equation} \label{dominf2}
\mathbb{P}\big(\underset{1 \leq q \leq N_1}{\min} |\Delta X_{T_{q}}|e^{-\beta_n\Delta_n} \leq  3 \underset{i \in {\mathcal A}_n^c}{\max} |\Delta_{i}^n X^c| \big) 
\end{equation} and
\begin{equation} \label{dominf3}
\mathbb{P}\big(\underset{1 \leq q \leq N_1}{\min} |\Delta X_{T_{q}}|e^{-\beta_n\Delta_n} \leq 3 \max_{1 \leq i \leq n} \beta_n \Delta_n |Z^{\beta}_{t_{i,n}}|  \big). 
\end{equation}
The term \eqref{dominf1} equals
\[\mathbb{E}\big[1 - \big(\mathbb{P}\big(|\xi_\nu| > 3 v_n \sqrt{\Delta_n} e^{\beta_n \Delta_n} \big)\big)^{N_1} \big] = 1 - \exp\big(-\lambda_n \mathbb{P}(|\xi_\nu| \leq 3 v_n \sqrt{\Delta_n} e^{\beta_n \Delta_n})\big)\]
and converges to 0 under the assumption $\lambda_n \mathbb{P}\big(|\xi_\nu| \leq   \Delta_n^{\frac{1}{2} - \varpi}\big) \rightarrow 0$. The term \eqref{dominf2} is dominated by 
\begin{equation} \label{dominf3b}
\mathbb{P}\big(\min_{1 \leq q \leq N_1} |\Delta X_{T_{q}}| e^{-\beta_n\Delta_n} \leq  3v_n \sqrt{\Delta_n} \big) + \mathbb{P}\big(v_n \sqrt{\Delta_n} \leq  \underset{i \in {\mathcal A}_n^c}{\max} |\Delta_{i}^n X^c| \big).
\end{equation}
The left hand side of \eqref{dominf3b} is equal to \eqref{dominf1} and converges to 0. According to \cite[Corollary 3.3]{mancini04}, for  $i \in \{1,...,n\}$,
\[\mathbb{P}\big(\Delta_{i}^n X^c > v_n \sqrt{\Delta_n} \big) \leq 2e^{-v_n^2/2{\bar{\sigma}}^2} .\]
The right hand side of \eqref{dominf3b} is dominated by 
\begin{align*}
\mathbb{E}\Big[\sum_{q=1}^{N_1} \mathbb{P}\big(|\Delta_{i}^n X^c| \geq  v_n \sqrt{\Delta_n} \big)\Big] &\leq \mathbb{E}\left[N_1\right]2e^{-v_n^2/2{\bar{\sigma}}^2} = 2 \lambda_n  e^{-v_n^2/2{\bar{\sigma}}^2}  \rightarrow 0.
\end{align*}
The term \eqref{dominf3} is dominated by 
\begin{equation} \label{dominf3a}
\mathbb{P}\big(\min_{1 \leq q \leq N_1} |\Delta X_{T_{q}}| e^{-\beta_n\Delta_n} \leq  3v_n \sqrt{\Delta_n} \big) + \mathbb{P}\big(v_n \sqrt{\Delta_n} \leq  \max_{1 \leq i \leq n} \beta_n \Delta_n |Z^{\beta}_{t_{i,n}}|  \big).
\end{equation}
The left hand side of \eqref{dominf3b} is equal to \eqref{dominf1} and converges to 0. The right hand side of \eqref{dominf3b} also converges to $0$ with the same argument as for \eqref{convsupz}.

\subsubsection*{Proof of Proposition \ref{estimatorn} under Assumption \ref{assumpestimatorn} (II)}
Since
\[\mathbb{P}\big(\underset{0 \leq i \leq n-k_n}{\max} |N_{\left(i+k_n\right)\Delta_n} - N_{i\Delta_n}| \geq 2 \big)  \lesssim \lambda_n^2 \Delta_n k^2_n \rightarrow 0,\]
therefore, we only need to prove the result on the set $\mathcal B_n = \{\underset{0 \leq i \leq n-k_n}{\max} |N_{\left(i+k_n\right)\Delta_n} - N_{i\Delta_n}| \leq 1 \}$. We need to show: 
\begin{equation} \label{convsupb}
\mathbb{P}\Big(\exists i \in {\mathcal A}_n, \frac{|\Delta^n_i X|}{\sqrt{\Delta_n}} > v_n \text{ and } \Delta_i^n X \Delta_{i+1}^n X < 0 \cap {\mathcal B}_n\Big) \rightarrow 0
\end{equation}
and
\begin{equation} \label{convsupb2}
\mathbb{P}\Big(\exists i \in \mathcal{A}^c_n, \frac{|\Delta^n_i X|}{\sqrt{\Delta_n}} < v_n \text{ or }  \Delta_i^n X \Delta_{i+1}^n X > 0 \cap \mathcal B_n\Big) \rightarrow 0.
\end{equation} 

We bound \eqref{convsupb} above by the sum of
$$
\mathbb{P}\big(\exists i \in {\mathcal A}_n, \frac{|\Delta^n_i X^c|}{\sqrt{\Delta_n}} > \frac{v_n}{2} \text{ and } \Delta_i^n X \Delta_{i+1}^n X < 0 \cap \mathcal B_n\big)
$$
and 
$$
\mathbb{P}\big(\exists i \in {\mathcal A}_n, \frac{|\Delta^n_i Z^{\beta}|}{\sqrt{\Delta_n}} > \frac{v_n}{2} \text{ and } \Delta_i^n X \Delta_{i+1}^n X < 0 \cap \mathcal B_n \big).
$$
The first term is bounded by
\[\mathbb{P}\big(\underset{1 \leq i \leq n}{\max} |\Delta^n_i X^c| > \frac{v_n \sqrt{\Delta_n}}{2}\big) \leq 2ne^{-v_n^2/8{\bar{\sigma}}^2}\]
and converges to 0. For the second term, we consider two cases: a jump occurs before $i\Delta_n$ or no such jump occurs, which leads us to further consider
\begin{equation}\label{convsupbba}
\mathbb{P}\big(\exists i \in {\mathcal A}_n, \frac{|\Delta^n_i Z^{\beta}|}{\sqrt{\Delta_n}} > \frac{v_n}{2} \text{, } \Delta_i^n X \Delta_{i+1}^n X < 0 \text{ and } \exists q \in \{1, \min\{k_n-2,i-1\}\}, i-q \in \mathcal A_n^c \} \cap \mathcal B_n\big)
\end{equation}
and 
\begin{equation} \label{convsupbbb}
\mathbb{P}\big(\exists i \in {\mathcal A}_n, \frac{|\Delta^n_i Z^{\beta}|}{\sqrt{\Delta_n}} > \frac{v_n}{2} \text{, } \Delta_i^n X \Delta_{i+1}^n X < 0 \text{ and } \forall q \in \{1, \min\{k_n-2,i-1\}\}, i-q \in {\mathcal A}_n \} \cap \mathcal B_n \big).
\end{equation}
For \eqref{convsupbba} since we work on $\mathcal B_n$, we have $i+1 \in {\mathcal A}_n$ hence 
 $\Delta_{i+1}^{n} Z^{\beta} = -(1-e^{-\beta_n \Delta_n}) Z^{\beta}_{t_{i,n}} = e^{-\beta_n \Delta_n} \Delta_i^n Z^{\beta}$
and \eqref{convsupbba} is dominated by the probability of the event 
\begin{align*}
& \{ \exists i \in {\mathcal A}_n,  -|\Delta_i^n X^c| |\Delta_{i+1}^n X^c| - |\Delta_i^n X^c| |\Delta_i^n Z^{\beta}| e^{-\beta_n\Delta_n} -| \Delta_{i+1}^n X^c| |\Delta_i^n Z^{\beta}| + \tfrac{v_n  \sqrt{\Delta_n}}{2}|\Delta_i^n Z^{\beta}|e^{-\beta_n\Delta_n} < 0\\
& \text{ and } \frac{|\Delta^n_i Z^{\beta}|}{\sqrt{\Delta_n}} > \frac{v_n}{2} \}  
\end{align*} 
equal to
\[
\mathbb{P}\big(\exists i \in {\mathcal A}_n,  \tfrac{v_n  \sqrt{\Delta_n}}{2}|\Delta_i^n Z^{\beta}|e^{-\beta_n\Delta_n} < 2\max_{1 \leq j \leq n} |\Delta^n_j X^c| |\Delta_{i}^n Z^{\beta}| + (\max_{1 \leq j \leq n} |\Delta^n_j X^c|)^2 \text{ and } \frac{|\Delta^n_i Z^{\beta}|}{\sqrt{\Delta_n}} > \frac{v_n}{2}     \big)
\]
and dominated by
\[\mathbb{P}\big(2\max_{1 \leq j \leq n} |\Delta^n_j X^c| > \tfrac{v_n  \sqrt{\Delta_n}}{4 }\big) + \mathbb{P}\big((\max_{1 \leq j \leq n} |\Delta^n_j X^c|)^2 > \frac{1}{2}(\tfrac{v_n  \sqrt{\Delta_n}}{2})^2\big) \rightarrow 0.
\]
 Concerning \eqref{convsupbbb}, we have, if $k_n \leq i+1$, we have that $|\Delta^n_i Z^{\beta}|$ is equal to
 \[\left(1-e^{-\beta_n \Delta_n}\right) e^{-\beta_n \Delta_n (k_n-2)} |Z^{\beta}_{\left(i-k_n+1\right)\Delta_n}| \leq \beta_n \Delta_n \underset{t \in \left[0,1\right]}{\sup} \int_{0}^{t} \int_{\mathbb{R}} |x| e^{-\beta_n\left(t -s\right)} \underline{p}\left(ds,dx\right)e^{-\beta_n \Delta_n (k_n-2)}.\]
 The inequality remains true if $i +1 < k_n$ as $\Delta_i^n Z^{\beta}$ is equal to 0 in this case. Thus, \eqref{convsupbbb} is dominated by 
 \begin{align*}
 \mathbb{P}\Big(\underset{t \in \left[0,1\right]}{\sup} \int_{0}^{t} \int_{\mathbb{R}} |x| e^{-\beta_n\left(t -s\right)} \underline{p}\left(ds,dx\right) > e^{\beta_n \Delta_n (k_n-2)} \tfrac{v_n}{\beta_n \sqrt{\Delta_n}} \Big) 
\leq 
\lambda_n \mathbb{P}\big(|\xi_{\nu}| > e^{\beta_n \Delta_n (k_n-1)} \tfrac{v_n}{\beta_n \sqrt{\Delta_n}}\big) \rightarrow 0
\end{align*}
using the same argument as for \eqref{convsupz} and the convergence \eqref{convsupb} follows.\\

We now turn to \eqref{convsupb2}. 
It suffices to show that both terms
$$
\mathbb{P}\big(\exists i \in {\mathcal A}_n^c, \frac{|\Delta^n_i X|}{\sqrt{\Delta_n}} \geq v_n \cap \mathcal B_n\big),
\;\;\text{and}\;\;
\mathbb{P}\big(\exists i \in {\mathcal A}_n^c, \Delta^n_i X \Delta^n_{i+1} X \geq 0 \cap \mathcal B_n\big)
$$
converge to $0$. The proof for the first term
is similar to the one of \eqref{convsup2}, the only difference being $\Delta^n_{i(n,q)} Z$ is equal to $(1-e^{-\beta_n \Delta_n})e^{-\beta_n \Delta_n (k_n-2)}Z_{t_{i(n,q)-k_n+1,n}} + e^{-\beta_n\left(t_{i(n,q),n}-T_q\right)} \Delta X_{T_q}$ if $k_n \leq i+1$ and $e^{-\beta_n\left(t_{i(n,q),n}-T_q\right)} \Delta X_{T_q}$ otherwise and that the term 
$\mathbb{P}( v_n \sqrt{\Delta_n} \leq 3 \max_{1 \leq i \leq n}\beta_n \Delta_n |Z^{\beta}_{t_{i,n}}|)$ needs to be replaced by 
$$\mathbb{P}\big( v_n \sqrt{\Delta_n} \leq 3 \max_{1 \leq i \leq n}\beta_n \Delta_n e^{-\beta_n \Delta_n k_n} |Z^{\beta}_{t_{i,n}}|  \big) \leq \lambda_n \mathbb{P}\big(|\xi_\nu| > e^{\beta_n \Delta_n (k_n-2)} \tfrac{v_n \sqrt{\Delta_n}}{\beta_n \Delta_n}  \big) \rightarrow 0.$$
For the second term,
it is sufficient to prove that $\Delta^n_{i\left(n,q\right)} X$ has the same sign as $\Delta X_{T_q}$ and that $\Delta^n_{i\left(n,q\right)+1} X$ has the opposite sign with probability one. We are thus led to show that    
\begin{equation} \label{convsupb2c1}
\mathbb{P}\big(\min_{1 \leq q \leq N_1} \Delta^n_{i\left(n,q\right)} X \Delta X_{T_q} < 0\big) \rightarrow 0
\end{equation}
and
\begin{equation} \label{convsupb2c2}
\mathbb{P}\big(\max_{1 \leq q \leq N_1}  \Delta^n_{i\left(n,q\right)+1} X \Delta X_{T_q} >  0\big) \rightarrow 0.
\end{equation}
We have that $\Delta^n_{i\left(n,q\right)} X \Delta X_{T_q}$ dominates 
\begin{align*}
- \max_{1 \leq q \leq N_1} |\Delta^n_{i\left(n,q\right)} X^c|  |\Delta X_{T_q}| & - (1-e^{-\beta_n \Delta_n}) e^{-\beta_n \Delta_n k_n} \underset{t \in \left[0,1\right]}{\sup} \int_{0}^{t} \int_{\mathbb{R}} |x| e^{-\beta_n\left(t -s\right)} \underline{p}\left(ds,dx\right) |\Delta X_{T_q}|  \\
&+ e^{-\beta_n \Delta_n} |\Delta X_{T_q}| \min_{1 \leq q \leq N_1} |\Delta X_{T_q}|.
\end{align*} 
We thus have that \eqref{convsupb2c1} is dominated by the probability of the event 
\begin{align*}
\Big\{e^{-\beta_n \Delta_n} \inf_{1 \leq q \leq N_1} |\Delta X_{T_q}| < &\sup_{1 \leq q \leq N_1} |\Delta^n_{i\left(n,q\right)} X^c| \\
&+ \left(1-e^{-\beta_n \Delta_n}\right) e^{-\beta_n \Delta_n k_n} \underset{t \in \left[0,1\right]}{\sup} \int_{0}^{t} \int_{\mathbb{R}} |x| e^{-\beta_n\left(t -s\right)} \underline{p}\left(ds,dx\right)\Big\}
\end{align*}
that converges to 0 if we use a similar proof than the one of \eqref{convsup2}. The proof of \eqref{convsupb2c2} is similar since no jump occurs in the interval $\left(i\left(n,q\right),i\left(n,q\right)+1\right]$ and 
\[
\Delta^n_{i\left(n,q\right)+1} Z^{\beta} = -\left(1-e^{-\beta_n \Delta_n}\right) Z^{\beta}_{i\left(n,q\right) \Delta_n}.
\]
The term \eqref{convsupb2c2} is then dominated by the probability of the event 
\begin{align*}
\Big\{e^{-\beta_n \Delta_n}(1-e^{-\beta_n\Delta_n}) &\min_{1 \leq q \leq N_1} |\Delta X_{T_q}| < \max_{1 \leq q \leq N_1} |\Delta^n_{i\left(n,q\right)} X^c| \\
&+ (1-e^{-\beta_n \Delta_n}) e^{-\beta_n \Delta_n (k_n+1)} \underset{t \in \left[0,1\right]}{\sup} \int_{0}^{t} \int_{\mathbb{R}} |x| e^{-\beta_n\left(t -s\right)} \underline{p}\left(ds,dx\right)\Big\}.
\end{align*}
The convergence to $0$ is obtained in the same way, except for the extra control of the terms 
\[ \mathbb{P}\big( \max_{1 \leq q \leq N_1} |\Delta^n_{i\left(n,q\right)} X^c| > v_n \sqrt{\Delta_n} (1-e^{-\beta_n \Delta_n}) \big) \leq \lambda_n e^{-v_n^2(1-e^{-\beta_n\Delta_n})^2/2\bar{\sigma}^2}\]
and  
$$
\mathbb{P}\Big(\underset{t \in \left[0,1\right]}{\sup} \int_{0}^{t} \int_{\mathbb{R}} |x| e^{-\beta_n\left(t -s\right)} \underline{p}\left(ds,dx\right) > e^{\beta_n \Delta_n (k_n+1)}v_n\sqrt{\Delta_n} \Big) \leq
\lambda_n \mathbb{P}\big(|\xi_\nu| > e^{\beta_n \Delta_n k_n}\Delta_n^{\frac{1}{2}-\varpi}\big)$$
that both converge to $0$. 
The proof of Proposition \ref{estimatorn} is complete.


\subsection{Proof of Theorem \ref{estimatorcj}}

\subsubsection*{Preparation for the proof}

In order to prove Theorem \ref{estimatorcj}, we start by giving an oracle estimator of $\beta_n$ when the jump times and their sizes are known.

\begin{proposition} \label{estimatorc} Work under Assumption \ref{assumpsigma} and \ref{betalambda}. Let $\beta_n \sqrt{\lambda_n \Delta_n}\rightarrow \infty$.
Define $\widehat{\beta}_n^{{\mathrm{oracle}}}$ via 
$$
\exp\big(-\Delta_n\widehat{\beta}_n^{{\mathrm{oracle}}}\big) = \max\Big\{1+\frac{\sum_{q \in \mathcal E_n} \mathrm{sgn}(\Delta X_{T_q}) \big(\Delta^n_{i\left(n,q\right)+1} X + 2\Delta_n\sum_{j=1}^{q-1}\Delta X_{T_j}\big)}{\sum_{q \in \mathcal E_n} \mathrm{sgn}(\Delta X_{T_q}) \Delta^n_{i\left(n,q\right)} X} {\bf 1}_{\{N_1 > 0\}},\Delta_n\Big\}, 
$$ 
with 
$ \mathcal E_n = \big\{q \in \{1,..,N_1\}, \; i\left(n,q\right)+1 \in {\mathcal A}_n \;\;\text{and}\;\;  i\left(n,q\right) < i\left(n,q+1\right)\big\}.$\\
\begin{itemize}
\item[i)] The following expansion holds on the set $\{N_1 > 0\}$: 
\[\frac{\widehat{\beta}_n^{{\mathrm{oracle}}} - \beta_n}{\beta_n} = \mathcal{M}_{n}^{oracle} + \mathcal V_{n}\mathcal J_{n}^T,\] 
with
\[\mathcal{M}_{n}^{oracle} = e^{\beta_n \Delta_n} \frac{\lambda_n}{\beta_n} \frac{(x \star \nu) \; (\mathrm{sgn}\left(x\right) \star \nu)}{|x| \star \nu} \big(\frac{e^{\beta_n \Delta_n}-1}{\beta_n \Delta_n} - \frac{\beta_n \Delta_n}{1-e^{-\beta_n \Delta_n}}  \big),\]
and $\mathcal V_{n} = (\mathcal V_n^{(i)})_{1 \leq i \leq 4} \in \R^4$ defined by
$$
\left\{
\begin{array}{ll}
    \mathcal V_{n}^{(1)}&=e^{\beta_n \Delta_n}\frac{\sqrt{\lambda_n}}{\sqrt{3} \beta_n |x| \star \nu} \big(\left(\mathrm{sgn}\left(x\right) \star \nu \right)^2 (|x|^2 \star \nu) + \left( x \star \nu \right)^2 - 2 (\mathrm{sgn}\left(x\right) \star \nu) (|x|^2 \star \nu)\big)^{1/2}, \\ 
    	\mathcal V_{n}^{(2)}  &= e^{\beta_n \Delta_n}\min\big\{\big(\tfrac{|x|^2 \star \nu}{\left(|x| \star \nu \right)^2} \frac{1}{2\beta_n}\frac{\left(1-e^{-2\beta_n \Delta_n}\right)}{2\beta_n \Delta_n}\big)^{1/2},\frac{\lambda_n}{\beta_n}\big\}, \\ 
       \mathcal V_{n}^{(3)} &=e^{\beta_n \Delta_n}\big(\frac{\beta_n \Delta_n}{1-e^{-\beta_n \Delta_n}}\big)\frac{\sqrt{\int_0^1 \sigma_s^2 ds}}{|x| \star \nu \sqrt{ \lambda_n}\beta_n \sqrt{\Delta_n}},\\ 
   	\mathcal V_{n}^{(4)}  &=  e^{\beta_n \Delta_n}\frac{\sqrt{\int_0^1 \sigma_s^2 ds} \sqrt{\Delta_n}}{|x| \star \nu \sqrt{\lambda_n}},\\
    \end{array}
\right.    
$$
and $\mathcal J_n = (\mathcal J_n^{(i)})_{1 \leq i \leq 4} \in \R^4$ is bounded in probability as $n \rightarrow \infty$.
\item[ii)] If $\lambda_n \rightarrow \infty$, then
$$
(\mathcal J_{n}^{(3)},\mathcal J_{n}^{(4)}) \rightarrow \mathcal{N}\left(0,\mathrm{Id}_{\R^2}\right)
$$
in distribution as $n \rightarrow \infty$.
\item[iii)] If $\lambda_n \rightarrow \infty$, $|x|^3 \star \nu < \infty$ and 
$(\mathrm{sgn}(x) \star \nu )^2 |x|^2 \star \nu + \left( x \star \nu \right)^2 - 2 \mathrm{sgn}\left(x\right) \star \nu |x|^2 \star \nu\neq 0$,
we have 
$$
(\mathcal J_{n}^{(1)}, \mathcal J_{n}^{(3)},J_{n}^{(4)}) \rightarrow \mathcal{N}(0,\mathrm{Id}_{\R^3})
$$ 
in distribution as $n \rightarrow \infty$.
\item[iv)] If $\beta_n/\lambda_n^2\rightarrow 0$, the conditions of $\mathrm{iii)}$ and $|x|^4 \star \nu < \infty$ hold together, we finally obtain  
$$
\mathcal J_{n} \rightarrow \mathcal{N}(0,\mathrm{Id}_{\R^4})
$$
in distribution as $n \rightarrow \infty$.
\end{itemize}
\end{proposition}

Proposition \ref{estimatorc} is the core of Theorem \ref{estimatorcj}. 

\subsubsection*{Proof of Proposition \ref{estimatorc}} \label{proof prop core}
\noindent {\it Step 1).} We first need three approximation results.
\begin{lemma} \label{approxlemma1}
We have
$$
\big|\sum_{q \notin \mathcal E_n} \mathrm{sgn}\left(\Delta X_{T_q}\right) \Delta^n_{i\left(n,q\right)} X\big| 
\lesssim \lambda_n^2 \Delta_n
$$
in probability.
\end{lemma}

\begin{proof}
We plan to use the decomposition
$$\sum_{q \notin \mathcal E_n} \mathrm{sgn}\left(\Delta X_{T_q}\right) \Delta^n_{i\left(n,q\right)} X  = \sum_{q=1}^{N_1} \mathrm{sgn}\left(\Delta X_{T_q}\right) \Delta^n_{i\left(n,q\right)} X {\bf 1}_{\{T_{q+1}-T_q < t_{i\left(n,q\right),n} - T_q + \Delta_n\}} = I+II,$$
with 
$$I  =\sum_{q=1}^{N_1}\mathrm{sgn}\left(\Delta X_{T_q}\right) \Delta^n_{i\left(n,q\right)} X^c {\bf 1}_{\{T_{q+1}-T_q < t_{i\left(n,q\right),n} - T_q + \Delta_n\}} \;\;\text{and}\;\;$$
$$II = \sum_{q=1}^{N_1} \mathrm{sgn}\left(\Delta X_{T_q}\right)\Delta^n_{i\left(n,q\right)} Z^{\beta} {\bf 1}_{\{T_{q+1}-T_q < t_{i\left(n,q\right),n} - T_q + \Delta_n\}}.$$
The term $I$ is centred and as $|I| \leq \sum_{q=1}^{N_1}  |\Delta^n_{i\left(n,q\right)} X^c| {\bf 1}_{\{T_{q+1}-T_q < t_{i\left(n,q\right),n} - T_q + \Delta_n\}}$, we have
\begin{align*}
\E[I^2]  &\leq  \, \mathbb{E}\big[(\sum_{q=1}^{N_1} |\Delta^n_{i\left(n,q\right)} X^c| {\bf 1}_{\{T_{q+1}-T_q < t_{i\left(n,q\right),n} - T_q + \Delta_n\}})^2 \big] \\
& \leq \mathbb{E}\big[(\sum_{i=1}^n |\Delta^n_{i} X^c| \Delta_i^n N {\bf 1}_{\{\Delta^n_i N \geq 2 \text{ or } \Delta^n_{i+1} N \geq 1\}})^2\big]\\
&= \sum_{i=1}^n \mathbb{E}\big[(\Delta^n_{i} X^c)^2\big] \mathbb{E}\big[(\Delta_i^n N)^2 {\bf 1}_{\{\Delta^n_i N \geq 2 \text{ or } \Delta^n_{i+1} N \geq 1}\big]
\end{align*}
with the convention $\Delta^n_{n+1} N  = 0$. Since 
$\mathbb{E}[(\Delta^n_{i} X^c)^2] 
\leq \bar{\sigma}^2 \Delta_n$ and $\mathbb{E}[(\Delta^n_i N)^2 ({\bf 1}_{\{\Delta^n_i N \geq 2\}} + {\bf 1}_{\{\Delta^n_{i+1} N \geq 1\}} )] \lesssim \lambda_n^2 \Delta_n^2$
we obtain
$$\E[I^2]  \lesssim (\lambda_n \Delta_n)^2.$$
In turn $I$ is of order $\lambda_n\Delta_n$ hence negligible. For the term $II$, we have $$|II| \leq \sum_{q=1}^{N_1} |\Delta^n_{i\left(n,q\right)} Z^{\beta}| {\bf 1}_{\{T_{q+1}-T_q < t_{i\left(n,q\right),n} - T_q + \Delta_n\}},$$
and 
\begin{align*}
|\Delta^n_{i\left(n,q\right)} Z^{\beta}| &= |-\left(1-e^{-\beta_n \Delta_n}\right)Z_{t_{i\left(n,q\right)-1,n}}^\beta + \sum_{j=1}^{\Delta^n_{i\left(n,q\right)}N} \Delta_{i\left(n,q\right)}^j Xe^{-\beta_n(t_{i\left(n,q\right),n} - T_{i\left(n,q\right)}^j)}|\\
&\leq \beta_n \Delta_n  |Z_{t_{i\left(n,q\right)-1,n}}^\beta| + \sum_{j=1}^{\Delta^n_{i\left(n,q\right)}N} |\Delta X^j_{i\left(n,q\right)}|, 
\end{align*}
where $\Delta X^j_i$ denotes the $j$th jump in the interval $\left(\left(i-1\right)\Delta_n, i\Delta_n\right]$ that occurs at time $T^j_i$.
First, we have that the term $\sum_{q=1}^{N_1} |Z^{\beta}_{t_{i\left(n,q\right)-1,n}} |{\bf 1}_{\{T_{q+1}-T_q < t_{i\left(n,q\right),n} - T_q + \Delta_n\}}$ is dominated by
$$\sum_{i=1}^n \int_{0}^{t_{i-1,n}} |x| e^{-\beta_n\left(t_{i-1,n}-t\right)} \underline{p}\left(dt,dx\right) \Delta^n_i N \big({\bf 1}_{\{\Delta^n_i N \geq 2\}} + {\bf 1}_{\{\Delta^n_{i+1} N \geq 1\}} \big).$$
Because of the independence of $\Delta^n_i N$ and $\Delta^n_{i+1} N$ conditional on $\mathcal{F}_{t_{i-1,n},n}$, we derive that its expectation is less than
$$\sum_{i=1}^n \mathbb{E}\big[\int_{0}^{t_{i-1,n}} |x| e^{-\beta_n\left(t_{i-1,n}-t\right)} \underline{p}(dt,dx)\big] \mathbb{E}\big[\Delta^n_i N({\bf 1}_{\{\Delta^n_i N \geq 2\}} + {\bf 1}_{\{\Delta^n_{i+1} N \geq 1\}} )\big].$$
Second, since $\mathbb{E}[\int_{0}^{t_{i-1,n}} |x| e^{-\beta_n\left(t_{i-1,n}-t\right)} \underline{p}(dt,dx)] \lesssim \lambda_n/\beta_n$
and also $\mathbb{E}[\Delta^n_i N ({\bf 1}_{\{\Delta^n_i N \geq 2\}} + {\bf 1}_{\{\Delta^n_{i+1} N \geq 1\}} )] \lesssim \lambda_n^2 \Delta_n^2$
we derive
\begin{equation} \label{boundGalileo}
\sum_{q=1}^{N_1} |Z^{\beta}_{t_{i\left(n,q\right)-1,n}} |{\bf 1}_{\{T_{q+1}-T_q < t_{i\left(n,q\right),n} - T_q + \Delta_n\}} \lesssim \frac{\lambda_n^3 \Delta_n}{\beta_n}
\end{equation}
in probability. In the same way, it is not difficult to see that 
\begin{align*}
\sum_{q=1}^{N_1}  \sum_{j=1}^{\Delta^n_{i\left(n,q\right)}N} |\Delta X^j_{i\left(n,q\right)}| {\bf 1}_{\{T_{q+1}-T_q < 2\Delta_n\}} &\leq \sum_{i=1}^n \Delta_i^n N \sum_{j=1}^{\Delta_i^n N} |\Delta X_i^j| ({\bf 1}_{\{\Delta^n_i N \geq 2\}} + {\bf 1}_{\{\Delta^n_{i+1} N \geq 1\}} )
\end{align*}
is of order $\lambda_n^2 \Delta_n$. The result of the lemma follows.
\end{proof}

\begin{lemma} \label{approxlemma2}
We have
$$\big| \sum_{q=1}^{N_1}\mathrm{sgn}(\Delta X_{T_q})\Delta^n_{i\left(n,q\right)}X - \int_{0}^{1} \int_{\mathbb{R}} |x| e^{-\beta_n((\lfloor t\Delta_n^{-1}\rfloor + 1)\Delta_n - t )} \underline{p}\left(dt,dx\right) \big| \lesssim \lambda_n^2 \Delta_n + \sqrt{\lambda_n \Delta_n}$$
in probability.
\end{lemma}
\begin{proof}
We plan to use the decomposition $\sum_{q=1}^{N_1}\mathrm{sgn}(\Delta X_{T_q})\Delta_{i\left(n,q\right)}X=I+II$, with
$$I  = \sum_{q=1}^{N_1} \mathrm{sgn}\left(\Delta X_{T_q}\right) \Delta^n_{i\left(n,q\right)} X^c,\;\;\text{and} \;\;
II  = \sum_{q=1}^{N_1} \mathrm{sgn}\left(\Delta X_{T_q}\right) \Delta^n_{i\left(n,q\right)} Z^{\beta}.$$
With the notation of Lemma \ref{approxlemma1}, we write
\begin{align*}
I &= \sum_{i=1}^n \sum_{j=1}^{\Delta^n_i N} \mathrm{sgn}(\Delta X^j_i)  \Delta^n_{i} X^c {\bf 1}_{\{\Delta_i^n N \geq 1\}} 
\end{align*}
and in the same way as for the proof of Lemma \ref{approxlemma1}, using the independence between $X^c$ and $N$, it is not difficult to see that $I$ is centred with variance of order $\lambda_n \Delta_n$. For the second term, we write $II=III+IV$, with
\begin{align*}
III & = \sum_{q=1}^{N_1} \mathrm{sgn}(\Delta X_{T_q}) \sum_{j=1}^{\Delta^n_{i\left(n,q\right)}N} \Delta X_{i\left(n,q\right)}^j e^{-\beta_n((\lfloor T_q\Delta_n^{-1} \rfloor+ 1)\Delta_n - T_{i\left(n,q\right)}^j)}, \\ 
IV & = - (1-e^{-\beta_n \Delta_n}) \sum_{q=1}^{N_1} \mathrm{sgn}(\Delta X_{T_q}) Z_{t_{i\left(n,q\right)-1,n}}^\beta.
\end{align*}
In the same way as in Lemma \ref{approxlemma1}, the term $III$ is equal to
$$ \sum_{q=1}^{N_1} \mathrm{sgn}(\Delta X_{T_q}) \big( \Delta X_{T_q} e^{-\beta_n((\lfloor T_q\Delta_n^{-1}\rfloor + 1)\Delta_n - T_q)} 
+ \sum_{\substack{j=1, \\T_{i\left(n,q\right)}^j \neq T_q}}^{\Delta^n_{i\left(n,q\right)} N } \Delta X^i_{q} e^{-\beta_n((\lfloor T_q\Delta_n^{-1} \rfloor + 1)\Delta_n - T^j_{i\left(n,q\right)})} {\bf 1}_{\{\Delta^n_{i\left(n,q\right)} N \geq 2\}} \big),$$
which is nothing but
$ \int_{0}^{1} \int_{\mathbb{R}} |x| e^{-\beta_n((\lfloor t\Delta_n^{-1}\rfloor + 1)\Delta_n - t)} \underline{p}\left(dt,dx\right)$
plus a remainder term of order $\lambda_n^2 \Delta_n$ in probability. Finally
\begin{align*}
| IV | & \leq 
\beta_n \Delta_n \sum_{q=1}^{N_1} \int_{0}^{t_{i\left(n,q\right)-1,n}} e^{-\beta_n (t_{i\left(n,q\right)-1,n}-t)}|x| \underline{p}(dt,dx) \\
& \leq \beta_n \Delta_n e^{\beta_n \Delta_n} \sum_{q=1}^{N_1} \int_{0}^{T_q^-} e^{-\beta_n (T_q^- -t )} |x| \underline{p}(dt,dx) \\
& = \beta_n \Delta_n e^{\beta_n \Delta_n} \int_{0}^1 \int_{0}^t |y|e^{-\beta_n (t-s)} \underline{p}(ds,dy) \underline{p}(dt,dx),
\end{align*}
and this term has expectation of order $\beta_n \Delta_n \lambda_n^2/\beta_n \lesssim \lambda_n^2 \Delta_n$.
\end{proof}

\begin{lemma} \label{approxlemma3}
We have
$$
\Big| \frac{\sum_{q \in \mathcal E_n} \mathrm{sgn}(\Delta X_{T_q}) \Delta^n_{i\left(n,q\right)} X}{\int_{0}^{1} \int_{\mathbb{R}} |x| e^{-\beta_n((\lfloor t\Delta_n^{-1}\rfloor + 1)\Delta_n - t)} \underline{p}\left(dt,dx\right)   } - 1 \Big| \lesssim \lambda_n \Delta_n + \sqrt{\Delta_n/\lambda_n}
$$
in probability.
\end{lemma}
\begin{proof}
If $\sup_n\lambda_n< \infty$, then
$\big(\int_{0}^{1} \int_{\mathbb{R}} |x| e^{-\beta_n((\lfloor t\Delta_n^{-1} \rfloor+ 1)\Delta_n - t)} \underline{p}(dt,dx)\big)^{-1}$
is bounded in probability. Otherwise, using $\int_{0}^1 e^{-\beta_n((\lfloor t\Delta_n^{-1} + 1)\Delta_n - t)} dt = \frac{1-e^{-\beta_n \Delta_n}}{\beta_n \Delta_n}$, we have
$$
\frac{\int_{0}^{1} \int_{\mathbb{R}} |x| e^{-\beta_n((\lfloor t\Delta_n^{-1}\rfloor + 1)\Delta_n - t)} \underline{p}(dt,dx)}{ \lambda_n |x| \star \nu(\frac{1-e^{-\beta_n \Delta_n}}{\beta_n \Delta_n})} \rightarrow 1
$$
as $n \rightarrow \infty$ and the result follows by applying Lemma \ref{approxlemma1} and Lemma \ref{approxlemma2}.
\end{proof}

\medskip

\noindent {\it Step 2).} We are ready to prove Proposition \ref{estimatorc}. Define the oracle slope
$$\widehat s_n^{\text{oracle}} = -\frac{\sum_{q \in \mathcal E_n} \mathrm{sgn}(\Delta X_{T_q}) \big(\Delta^n_{i\left(n,q\right)+1} X+2\Delta_n \sum_{j=1}^{q-1}\Delta X_{T_j}\big)}{\sum_{q \in \mathcal E_n} \mathrm{sgn}(\Delta X_{T_q}) \Delta^n_{i\left(n,q\right)} X} {\bf 1}_{\{N_1 > 0\}}.$$
Using the canonical decomposition $X = X^c+Z^\beta$ and the fact that for $i \in \mathcal A_n$, we have 
$$-\Delta_i^n Z^{\beta} = (1-e^{-\beta_n \Delta_n}) Z^{\beta}_{t_{i-1,n},n}  
= (1-e^{-\beta_n \Delta_n})(\Delta^n_{i-1} Z^{\beta} + Z^{\beta}_{t_{i-2},n}),$$
and we write 
$$\widehat{s}_n^{{\mathrm{oracle}}} = I+II+III,$$ with
\begin{align*}
I & = -\frac{\sum_{q \in \mathcal E_n} \mathrm{sgn}(\Delta X_{T_q}) \Delta^n_{i\left(n,q\right)+1} X^c}{\sum_{q \in \mathcal E_n} \mathrm{sgn}(\Delta X_{T_q}) \Delta^n_{i\left(n,q\right)} X} {\bf 1}_{\{N_1 > 0\}},\\ 
II & = (1-e^{-\beta_n \Delta_n})\frac{\sum_{q \in \mathcal E_n} \mathrm{sgn}(\Delta X_{T_q}) \Delta^n_{i\left(n,q\right)} Z^{\beta}}{\sum_{q \in \mathcal E_n}  \mathrm{sgn}(\Delta X_{T_q})\Delta^n_{i\left(n,q\right)} X} {\bf 1}_{\{N_1 > 0\}}, \\
III &   = (1-e^{-\beta_n \Delta_n}) \frac{\sum_{q \in \mathcal E_n}  \mathrm{sgn}(\Delta X_{T_q})\big(Z^{\beta}_{t_{i\left(n,q\right)-1,n}} - \frac{2\Delta_n}{1-e^{-\beta_n \Delta_n}} \sum_{j=1}^{q-1}\Delta X_{T_j}\big)}{\sum_{q \in \mathcal E_n} \mathrm{sgn}(\Delta X_{T_q}) \Delta^n_{i\left(n,q\right)} X } {\bf 1}_{\{N_1 > 0\}}.
\end{align*}
We  study the convergence of each term separately.
\medskip

\noindent {\it Step 3)}. The  term $I$. From the proof of Lemma \ref{approxlemma1}, we readily have
$$\big|\sum_{q \notin \mathcal E_n}\mathrm{sgn}(\Delta X_{T_q}) \Delta^n_{i\left(n,q\right)+1} X^c
\big| \lesssim \lambda_n \Delta_n$$
in probability, so we shall replace the sum in $q$ over $\mathcal E_n$ by the sum in $q$ over $\{1,\ldots, N_1\}$ in the following. By Lemma \ref{approxlemma3}, we derive
\begin{equation} \label{first approx brownianpart}
I = IV(1+(\lambda_n\Delta_n+\sqrt{\Delta_n/\lambda_n})\mathcal R_n^{(1)}) +\Delta_n\mathcal R_n^{(2)},
\end{equation}
with
$$IV = -\frac{\sum_{q=1}^{N_1}  \mathrm{sgn}(\Delta X_{T_q}) \Delta^n_{i\left(n,q\right)+1} X^c }{\int_{0}^{1} \int_{\mathbb{R}} |x| e^{-\beta_n((\lfloor t\Delta_n^{-1} \rfloor+ 1)\Delta_n - t)} \underline{p}(dt,dx)}$$
and where both $\mathcal R_n^{(1)}$ and $\mathcal R_n^{(2)}$ are bounded in probability. Next, denoting by $\Delta X_i^j$ the $j$-th jump in $[(i-1)\Delta_n,i\Delta_n]$, we have 
\begin{align*}
 \sum_{q=1}^{N_1}  \mathrm{sgn}(\Delta X_{T_q}) \Delta^n_{i\left(n,q\right)+1} X^c 
& = \sum_{i=2}^n \mathrm{sgn}( \Delta X^1_{i-1})\Delta^n_i X^c  {\bf 1}_{\{\Delta^n_{i-1} N = 1\}} +  \sum_{i=2}^n\sum_{j=2}^{\Delta^n_{i-1} N}\mathrm{sgn}( \Delta X^j_i) \Delta^n_i X^c  {\bf 1}_{\{\Delta^n_i N \geq 2\}} \\
& = \sum_{i=2}^n \mathrm{sgn}( \Delta X^1_{i-1})\Delta^n_i X^c  {\bf 1}_{\{\Delta^n_{i-1} N = 1\}} +\lambda_n\Delta_n \mathcal R_n^{(3)},
\end{align*}
where $\mathcal R_n^{(3)}$ is bounded in probability, using the same argument as in Lemma \ref{approxlemma2}. It is not difficult to see that $ \sum_{i=2}^n \mathrm{sgn}( \Delta X^1_{i-1})\Delta^n_i X^c  {\bf 1}_{\{\Delta^n_{i-1} N \geq 1\}}$ is centred with variance $\lambda_n\Delta_n\int_{\Delta_n}^1\sigma_s^2ds$ up to an error of order $(\lambda_n\Delta_n)^2$.  Therefore, when $\sup_n \lambda_n<\infty$, we have that $IV$ is of order $\sqrt{\Delta_n}$ and of order $\sqrt{\Delta_n/\lambda_n}$ otherwise, using that $\int_0^1 \int_{\mathbb{R}} |x| e^{-\beta_n((\lfloor t\Delta_n^{-1}\rfloor + 1)\Delta_n - t)} \underline{p}\left(dt,dx\right)$ is equivalent to $\lambda_n \Delta_n |x| \star \nu\left(\frac{1-e^{-\beta_n\Delta_n}}{\beta_n \Delta_n}\right)$. We thus obtain the decomposition 
\begin{align*} \label{decomp finale I}
I &= \left(\frac{\beta_n \Delta_n}{1-e^{-\beta_n\Delta_n}}\right)\frac{\sqrt{\Delta_n \int_0^1 \sigma_s^2ds}}{\sqrt{\lambda_n}|x| \star \nu}  \mathcal J_{3}^{(n)} \numberthis\\
&= e^{-\beta_n\Delta_n} \beta_n \Delta_n \mathcal{V}_3^{(n)}J_3^{(n)},
\end{align*}
where $\mathcal J_{3}^{(n)}$ is bounded in probability and $\mathcal{V}_3^{(n)}$ is defined in the statement of Proposition \ref{estimatorc}. We investigate further the convergence of $\mathcal J_{3}^{(n)}$. Define 
$$V_n^2=\sum_{i=2}^n \E\big[(\Delta_i^nX^c)^2{\bf 1}_{\{\Delta_{i-1}^nN = 1\}}\,|\mathcal F_{t_{i-1,n}}^W\big].$$
Clearly, $\sum_{i=2}^n \mathrm{sgn}( \Delta X^1_{i-1})\Delta^n_i X^c  {\bf 1}_{\{\Delta^n_{i-1} N = 1\}}$ is centred and we claim that for every $\eta >0$:
\begin{equation} \label{lindeberg}
\sum_{i = 2}^n \E\big[V_n^{-2}(\Delta_i^nX^c)^2{\bf 1}_{\{\Delta_{i-1}^nN = 1\}}{\bf 1}_{\{V_n^{-1}|\Delta_i^nX^c| \geq \eta\}}{\bf 1}_{\{N_1\geq 1\}}\big] \rightarrow 0
\end{equation}
as $n \rightarrow \infty$.
Indeed, applying successively Cauchy-Schwarz's, Markov's and Burckolder-Davis-Gundy's inequality, we obtain
\begin{align*}
\mathbb{E}\big[(\Delta^n_i X^c)^2 {\bf 1}_{\{V_n^{-1}|\Delta^n_i X^c| \geq \eta\}} \,| \mathcal F^N_1\big]  &\leq  \mathbb{E}\big[(\Delta^n_i X^c)^4\big]^{1/2} \mathbb{P}\big[|\Delta^n_i X^c| \geq \eta V_n \,| \mathcal F_1^N\big]^{1/2} \\
& \leq \mathbb{E}\big[(\Delta^n_i X^c)^4\big]^{1/2} \frac{\mathbb{E}[(\Delta^n_i X^c)^2]^{1/2}}{\eta V_n}  \lesssim \eta^{-1}V_n^{-1}\Delta_n^{3/2}.
\end{align*} 
Next $V_n^2 = \sum_{i = 2}^n{\bf 1}_{\{\Delta_{i-1}^nN = 1\}}\int_{t_{i-1,n}}^{t_{i,n}}\sigma_s^2ds \geq \underline{\sigma}^2\Delta_n  \sum_{i = 2}^n{\bf 1}_{\{\Delta_{i-1}^nN =1\}} =  \underline{\sigma}^2\Delta_n \underline V_n^2$ say.
Summing up and taking expectation, it follows that
\begin{align*}
\sum_{i = 2}^n \E\big[V_n^{-2}(\Delta_i^nX^c)^2{\bf 1}_{\{\Delta_{i-1}^nN = 1\}}{\bf 1}_{\{V_n^{-1}|\Delta_i^nX^c| \geq \eta\}}\big] & \lesssim \eta^{-1}\Delta_n^{3/2}\sum_{i = 2}^n \E\big[V_n^{-3} {\bf 1}_{\{\Delta_{i-1}^nN = 1\}}{\bf 1}_{\{N_1\geq 1\}}\big] \\ 
  &\lesssim \E\big[\underline{V}_n^{-1}{\bf 1}_{\{N_1\geq 1\}}\big] \leq \E\big[(\underline{V}_n^2)^{-1}{\bf 1}_{\{N_1\geq 1\}}\big]^{1/2}
\end{align*}
by Jensen's inequality. Since $\underline{V}_n^2$ has a Binomial distribution with parameters $(n-1, \lambda_n\Delta_n e^{-\lambda_n\Delta_n})$, we have that $\underline{V}_n^2 \rightarrow \infty$ in probability since $\lambda_n \rightarrow \infty$ and is bounded below on $\{N_1\geq 1\}$ which has probability that converges to one, the Lindeberg condition \eqref{lindeberg} follows by dominated convergence and we further infer
$$\frac{1}{V_n}\sum_{i=2}^n \Delta^n_i X^c \mathrm{sgn}(\Delta^1_j X) {\bf 1}_{\{\Delta^n_{i-1} N = 1\}} \rightarrow \mathcal N(0,1)$$
in distribution as $n \rightarrow \infty$. Observing that $\frac{V_n^2}{\lambda_n\Delta_n} \rightarrow \int_0^1 \sigma_s^2ds$ in probability, in view of \eqref{first approx brownianpart}, we conclude
$$
\frac{1}{\sqrt{\lambda_n \Delta_n \int_0^1 \sigma^2_s ds}}\sum_{q=1}^{N_1} \mathrm{sgn}(\Delta X_{T_q})\Delta^n_{i\left(n,q\right)+1} X^c   \rightarrow \mathcal{N}\left(0,1\right)
$$
in distribution as $n\rightarrow \infty$ and likewise for $\mathcal J_{3}^{(n)}$ in view of \eqref{decomp finale I}.\\

\noindent {\it Step 4).} The term $II$. We write, using the proof of Lemma \ref{approxlemma1} and Lemma \ref{approxlemma3}, 
\begin{align*}
 & (1-e^{-\beta_n \Delta_n})^{-1}II \\
  = & \,1-\frac{\sum_{q \in \mathcal E_n}  \mathrm{sgn}(\Delta X_{T_q}) \Delta^n_{i\left(n,q\right)} X^c}{\sum_{q \in \mathcal E_n}  \mathrm{sgn}(\Delta X_{T_q}) \Delta^n_{i\left(n,q\right)}X } \\
 =&\, 1-
\frac{\sum_{q \in \mathcal E_n}  \mathrm{sgn}(\Delta X_{T_q}) \Delta^n_{i\left(n,q\right)} X^c}{\int_{0}^{1} \int_{\mathbb{R}} |x| e^{-\beta_n((\lfloor t\Delta_n^{-1}\rfloor + 1)\Delta_n - t)} \underline{p}\left(dt,dx\right)}\big(1+\max\{\lambda_n\Delta_n,\sqrt{\lambda_n/\Delta_n}\}\mathcal R_n^{(1)}\big)\\
 =&\, 1-
(\frac{\sum_{q =1}^{N_1}  \mathrm{sgn}(\Delta X_{T_q}) \Delta^n_{i\left(n,q\right)} X^c}{\int_{0}^{1} \int_{\mathbb{R}} |x| e^{-\beta_n((\lfloor t\Delta_n^{-1}\rfloor + 1)\Delta_n - t)} \underline{p}\left(dt,dx\right)} +  \Delta_n \mathcal R_n^{(2)})\big(1+\max\{\lambda_n\Delta_n,\sqrt{\lambda_n/\Delta_n}\}\mathcal R_n^{(1)}\big)  
\end{align*}
where $\mathcal R_n^{(1)}$ and $\mathcal R_n^{(2)}$ are bounded in probability. By Step 2), we know that 
\begin{align*}
\frac{\sum_{q =1}^{N_1}  \mathrm{sgn}(\Delta X_{T_q}) \Delta^n_{i\left(n,q\right)} X^c}{\int_{0}^{1} \int_{\mathbb{R}} |x| e^{-\beta_n((\lfloor t\Delta_n^{-1}\rfloor + 1)\Delta_n - t)} \underline{p}\left(dt,dx\right)} &= \left(\frac{\beta_n \Delta_n}{1-e^{-\beta_n\Delta_n}}\right)\tfrac{\sqrt{\Delta_n\int_0^1\sigma_s^2ds}}{\sqrt{\lambda_n}|x|\star \nu} \mathcal U^{(n)}\\
&= e^{-\beta_n\Delta_n}\left(\frac{\beta_n \Delta_n}{1-e^{-\beta_n\Delta_n}}\right) \mathcal{V}_4^{(n)},
\end{align*}
where $\mathcal U^{(n)}$ is bounded in probability and asymptotically normal if $\lambda_n \rightarrow \infty$ and $\mathcal{V}_4^{(n)}$ is defined in the statement of Proposition \ref{estimatorc}. Finally, we have proved
$$II = (1-e^{-\beta_n\Delta_n})\big(1+e^{-\beta_n\Delta_n}\left(\frac{\beta_n \Delta_n}{1-e^{-\beta_n\Delta_n}}\right) \mathcal{V}_4^{(n)}\mathcal J_4^{(n)}\big)$$
where $\mathcal J_4^{(n)}$ is bounded in probability and asymptotically normal if $\lambda_n \rightarrow \infty$.\\

\noindent {\it Step 4').} It is easily shown that $\sum_{i = 2}^n\E\big[(\Delta_i^nX^c)^2{\bf 1}_{\{\Delta_{i-1}N=1\}}{\bf 1}_{\{\Delta_i^nN=1\}}{\bf 1}_{\{N_1\geq 1\}}\,| \mathcal F_{t_{i-1,n}}\big]\rightarrow 0$ if $\lambda_n \rightarrow \infty$ as $n \rightarrow \infty$, so we actually have from Step 3) and Step 4) the joint convergence $(\mathcal J_{n}^{(3)},\mathcal J_{n}^{(4)}) \rightarrow \mathcal N(0,\mathrm{Id}_{\R^2})$ in distribution as $\lambda_n \rightarrow \infty$.\\

%
\noindent {\it Step 5).} The term $III$. By the proof of Lemma \ref{approxlemma1} and Lemma \ref{approxlemma3}, we have
\begin{align*}
 & (1-e^{-\beta_n\Delta_n})^{-1}III \\
  = &\,  \frac{\sum_{q \in \mathcal E_n}  \mathrm{sgn}(\Delta X_{T_q}) \big(Z^{\beta}_{t_{i\left(n,q\right)-1,n}} - \frac{2\Delta_n}{1-e^{-\beta_n \Delta_n}} \sum_{j=1}^{q-1}\Delta X_{T_j}\big)}{\int_{0}^{1} \int_{\mathbb{R}} |x| e^{-\beta_n((\lfloor t\Delta_n^{-1} \rfloor+ 1)\Delta_n - t)} \underline{p}(dt,dx)}\big(1+\max\{\lambda_n\Delta_n,\sqrt{\Delta_n/\lambda_n}\}\mathcal R_n^{(1)}\big) \\
  = & \, \Big(\frac{\sum_{q =1}^{N_1}  \mathrm{sgn}(\Delta X_{T_q}) \big(Z^{\beta}_{t_{i\left(n,q\right)-1,n}} - \frac{2\Delta_n}{1-e^{-\beta_n \Delta_n}} \sum_{j=1}^{q-1}\Delta X_{T_j}\big)}{\int_{0}^{1} \int_{\mathbb{R}} |x| e^{-\beta_n((\lfloor t\Delta_n^{-1} \rfloor+ 1)\Delta_n - t)} \underline{p}(dt,dx)}+\tfrac{\lambda_n^2}{\beta_n}\Delta_n \mathcal R_n^{(2)}\Big)\big(1+\max\{\lambda_n\Delta_n,\sqrt{\Delta_n/\lambda_n}\}\mathcal R_n^{(1)}\big)
\end{align*}
where $\mathcal R_n^{(1)}$ and $\mathcal R_n^{(2)}$ are bounded in probability. Indeed, in the same way as for Lemma \ref{approxlemma1}, we have
\begin{align*}
 & \big|\sum_{q \notin \mathcal E_n} \mathrm{sgn}(\Delta X_{T_q}) \big(Z^{\beta}_{t_{i\left(n,q\right)-1,n}} - \tfrac{2\Delta_n}{1-e^{-\beta_n \Delta_n}} \sum_{j=1}^{q-1}\Delta X_{T_j}\big)\big|  \\
 \leq & \,\sum_{q \notin \mathcal E_n}\big(|Z^{\beta}_{t_{i\left(n,q\right)-1,n}}| + \tfrac{2\Delta_n}{1-e^{-\beta_n \Delta_n}} \sum_{j=1}^{q-1}|\Delta X_{T_j}|\big) {\bf 1}_{\{T_{q+1}-T_q < 2\Delta_n\}} \lesssim \frac{\lambda_n^3\Delta_n}{\beta_n}
 \end{align*}
in probability, as follows from \eqref{boundGalileo} and the computations of Lemma \ref{approxlemma1}. When exactly one jump occurs in $(i\left(n,q\right)-1,i\left(n,q\right)]$, we have  $Z^{\beta}_{t_{i\left(n,q\right)-1,n}} = e^{-\beta_n(t_{i\left(n,q\right)-1,n} - T_q)} Z^{\beta}_{T_{q^-}}$ and 
\begin{align*}
& \sum_{q =1}^{N_1}  \mathrm{sgn}(\Delta X_{T_q}) \big(Z^{\beta}_{t_{i\left(n,q\right)-1,n}} - \tfrac{2\Delta_n}{1-e^{-\beta_n \Delta_n}} \sum_{j=1}^{q-1}\Delta X_{T_j}\big) \\
= & \int_{0 \leq s < t \leq 1}\int_{\mathbb{R}^2}   y\, \mathrm{sgn}(x)(e^{-\beta_n(t-s)} e^{-\beta_n(\lfloor t\Delta_n^{-1} \rfloor \Delta_n-t)}-\tfrac{2 \Delta_n}{1-e^{-\beta_n\Delta_n}}) \underline{p}(ds,dy) \underline{p}(dt,dx)  + \tfrac{\lambda_n^3\Delta_n}{\beta_n}\mathcal R_n^{(3)},
\end{align*}
where the remainder term $\mathcal R_n^{(3)}$ is bounded in probability and accounts for the case where more than one jump occurs in the intervals $(i\left(n,q\right)-1,i\left(n,q\right)]$.
By Fubini's theorem, the main term splits into $\mathcal M_n^{(1)}+ \mathcal M_n^{(2)}+ \mathcal M_n^{(3)}$, with 
\begin{align*}
\mathcal M_n^{(1)} & = \lambda_n^2 (x \star \nu) \;  (\mathrm{sgn}(x) \star \nu) \int_{0 \leq s < t \leq 1}\big(e^{-\beta_n(t-s)} e^{-\beta_n(\lfloor t\Delta_n^{-1} \rfloor \Delta_n-t)}-\tfrac{2 \Delta_n}{1-e^{-\beta_n\Delta_n}}\big) dt ds\\
& =  \tfrac{\lambda_n^2}{\beta_n} \,(x \star \nu) \; (\mathrm{sgn}\left(x\right) \star \nu) \big(\tfrac{e^{\beta_n \Delta_n}-1}{\beta_n \Delta_n} - \Delta_n \tfrac{1-e^{-\beta_n}}{1-e^{-\beta_n \Delta_n}}- \tfrac{\beta_n \Delta_n}{1-e^{-\beta_n \Delta_n}}  \big), \\
\mathcal M_n^{(2)} & = \lambda_n \int_{0}^1 \int_{\mathbb{R}} g_{n}^{(1)}(t,x)(\underline{p}-\lambda_n \underline{q})(dt,dx),\\
\mathcal M_n^{(3)} & = \int_{\left[0,1\right]^2 \times \mathbb{R}^2}  g_n^{(2)}\left(t,x,s,y\right) \left(\underline{p}-\lambda_n \underline{q}\right)\left(ds,dy\right)\left(\underline{p}-\lambda_n \underline{q}\right)\left(dt,dx\right),
\end{align*}
with
\begin{align*}g_{n}^{(1)}(t,x) &=  x  \, (\mathrm{sgn}(y) \star \nu) \int_{0 \leq s < t \leq 1}(e^{-\beta_n(s-t)} e^{-\beta_n(\lfloor s\Delta_n^{-1} \rfloor\Delta_n-s)}-\tfrac{2 \Delta_n}{1-e^{-\beta_n\Delta_n}}) ds \\
&+  (y \star \nu) \, \mathrm{sgn}(x) \int_{0 \leq s < t \leq 1}(e^{-\beta_n(t-s)} e^{-\beta_n(\lfloor t\Delta_n^{-1} \rfloor \Delta_n-t)}-\tfrac{2 \Delta_n}{1-e^{-\beta_n\Delta_n}}) ds 
\end{align*}
and
$$
g_n^{(2)}\left(t,x,s,y\right) =y \, \mathrm{sgn}(x){\bf 1}_{\{t > s\}}\big(e^{-\beta_n(t-s)} e^{-\beta_n(\lfloor t\Delta_n^{-1} \rfloor \Delta_n-t)}-\tfrac{2 \Delta_n}{1-e^{-\beta_n\Delta_n}}\big).
$$
By standard yet tedious computations, evaluating centred terms and their variances in terms of asymptotics in $n$, we can infer from the previous representation that
\begin{align*}
&\frac{\sum_{q =1}^{N_1}  \mathrm{sgn}(\Delta X_{T_q}) \big(Z^{\beta}_{t_{i\left(n,q\right)-1,n}} - \frac{2\Delta_n}{1-e^{-\beta_n \Delta_n}} \sum_{j=1}^{q-1}\Delta X_{T_j}\big)}{\int_{0}^{1} \int_{\mathbb{R}} |x| e^{-\beta_n((\lfloor t\Delta_n^{-1} \rfloor+ 1)\Delta_n - t)} \underline{p}(dt,dx)} \\
 = &\, \frac{\mathcal M_n^{(1)}}{\lambda_n (|x|\star \nu)\tfrac{1-e^{-\beta_n\Delta_n}}{\beta_n\Delta_n} } +\frac{\lambda_n\int_0^1 \int_{\mathbb{R}}g_n^{(3)}(t,x)(\underline{p}-\lambda_n\underline{q})(dt,dx)}{\int_{0}^{1} \int_{\mathbb{R}} |x| e^{-\beta_n((\lfloor t\Delta_n^{-1} \rfloor+ 1)\Delta_n - t)} \underline{p}(dt,dx)}\\
 &+\tfrac{\beta_n\Delta_n}{1-e^{-\beta_n\Delta_n}}\min\big\{\sqrt{\tfrac{x^2\star\nu}{2(|x|\star\nu)^2\beta_n}\frac{1-e^{-2\beta_n\Delta_n}}{2\beta_n\Delta_n}},\lambda_n\beta_n^{-1}\big\}\mathcal J_n^{(2)},
\end{align*}
where $\mathcal J_n^{(2)}$ is bounded in probability and $g_n^{(3)}(t,x) = -\lambda_n^{-2}\frac{\mathcal M_n^{(1)}}{|y|\star\nu\left(\frac{1-e^{-\beta_n\Delta_n}}{\beta_n\Delta_n}\right)}|x|e^{-\beta_n((\lfloor t\Delta_n^{-1}\rfloor+1)\Delta_n-t)}+g_n^{(1)}(t,x)$. It is not difficult to check that
$$\frac{\lambda_n\int_0^1g_n^{(3)}(t,x)(\underline{p}-\lambda_n\underline{q})(dt,dx)}{\int_{0}^{1} \int_{\mathbb{R}} |x| e^{-\beta_n((\lfloor t\Delta_n^{-1} \rfloor+ 1)\Delta_n - t)} \underline{p}(dt,dx)} = e^{-\beta_n \Delta_n}\mathcal V_n^{(1)}\mathcal J_n^{(1)}$$
where $\mathcal V_n^{(1)}$ is defined in the statement of Proposition \ref{estimatorc} and $\mathcal J_n^{(1)}$ is bounded in probability. It follows that
\begin{align*}
&\frac{1-e^{-\beta_n\Delta_n}}{\beta_n\Delta_n} \frac{\sum_{q =1}^{N_1}  \mathrm{sgn}(\Delta X_{T_q}) \big(Z^{\beta}_{t_{i\left(n,q\right)-1,n}} - \frac{2\Delta_n}{1-e^{-\beta_n \Delta_n}} \sum_{j=1}^{q-1}\Delta X_{T_j}\big)}{\int_{0}^{1} \int_{\mathbb{R}} |x| e^{-\beta_n((\lfloor t\Delta_n^{-1} \rfloor+ 1)\Delta_n - t)} \underline{p}(dt,dx)} \\
& = \, \frac{\mathcal M_n^{(1)}}{\lambda_n (|x|\star \nu)} +  e^{-\beta_n \Delta_n}\mathcal V_n^{(1)}\mathcal J_n^{(1)}+\min\big\{\sqrt{\tfrac{x^2\star\nu}{2(|x|\star\nu)^2\beta_n}\frac{1-e^{-2\beta_n\Delta_n}}{2\beta_n\Delta_n}},\lambda_n\beta_n^{-1}\big\}\mathcal J_n^{(2)}.
\end{align*}
We eventually obtain the decomposition 
\begin{align*}
III & = \beta_n\Delta_n \underline{\mathcal M}_n  + \beta_n\Delta_n e^{-\beta_n \Delta_n}\mathcal V_n^{(1)}\mathcal J_n^{(1)}+\beta_n\Delta_ne^{-\beta_n \Delta_n}\mathcal V_n^{(2)}\mathcal J_n^{(2)}
\end{align*}
with $\underline{\mathcal M}_n = \tfrac{\lambda_n}{\beta_n} \,\frac{(x \star \nu) \; (\mathrm{sgn}\left(x\right) \star \nu)}{|x| \star \nu} \Big(\tfrac{e^{\beta_n \Delta_n}-1}{\beta_n \Delta_n} - \Delta_n \tfrac{1-e^{-\beta_n}}{1-e^{-\beta_n \Delta_n}}- \tfrac{\beta_n \Delta_n}{1-e^{-\beta_n \Delta_n}}  \Big)$ and $\mathcal V_n^{(2)}$ defined in the statement of Proposition \ref{estimatorc}.\\ 

\noindent {\it Step 5').} Using classical results of \cite[Theorem 3]{peccati08}, one can further investigate the asymptotic distribution of $(\mathcal J_n^{(1)}, \mathcal J_n^{(2)})$ under the additional 
assumption  $|x|^3 \star \nu < \infty$ assuring the existence of $\int_{0}^1 \int_{\mathbb{R}}g^3_n\left(t,x\right) \nu\left(dx\right) dt$ and the condition 
\[\left(\mathrm{sgn}\left(x\right) \star \nu \right)^2(|x|^2 \star \nu) + \left(x \star \nu \right)^2 - 2 (\mathrm{sgn}\left(x\right) \star \nu) (|x|^2 \star \nu) \neq 0,\]
otherwise the term $\mathcal V_n^{(1)}$ becomes negligible. We then have that $\mathcal J_n^{(1)}$ is asymptotically normal, and under the stronger assumption
$|x|^4 \star \nu < \infty$, we even have the convergence of $(\mathcal J_n^{(1)}, \mathcal J_n^{(2)})$ towards a standard two-dimensional Gaussian distribution. We omit the details.\\

\noindent {\it Step 6).}
Combining Step 2) and the results of Steps 3), 4) and 5), we obtain
$$
\widehat{s}_n^{\text{oracle}} = \left(1-e^{-\beta_n \Delta_n}\right) + \beta_n \Delta_n \underline{\mathcal M}_n + \beta_n \Delta_n e^{-\beta_n\Delta_n}\mathcal V_{n}^T \mathcal J_{n}
$$
with the notation introduced in the statement of Proposition \ref{estimatorc}.
%
Now, let $\widehat{\beta}_n^{\text{oracle}} = -\frac{1}{\Delta_n}\log(\max\{1-\widehat{s}_n^{\text{oracle}},\Delta_n\})$.
For $n$ large enough,
a first-order Taylor's expansion entails
\[\widehat{\beta}_n^{\text{oracle}} = \beta_n + \beta_n e^{\beta_n \Delta_n} \left(\underline{\mathcal M}_{n} + e^{-\beta_n\Delta_n}\mathcal V_{n}^T \underline{\mathcal J}_{n}\right)  \]
where the random vector  $\underline{\mathcal J}_{n}$ has the same asymptotic properties as $\mathcal J_{n}$. Setting $\mathcal{M}^{oracle}_n = e^{\beta_n \Delta_n} \underline{\mathcal M}_{n}$, and checking that all the terms have the right order, we obtained the desired result and Proposition \ref{estimatorc} is proved.

\subsubsection*{Completion of proof of Theorem \ref{estimatorcj}}
We start by a simple approximation result.
\begin{lemma} \label{lemma sgn}
We have
$$
\mathbb{P}\big( \mathrm{sgn}(\Delta X_{T_q})  = \mathrm{sgn}(\Delta_{i\left(n,q\right)}^n X)\;\;\text{for every}\;\;1 \leq q \leq N_1\big) \rightarrow 1.
$$
\end{lemma}

\begin{proof}
Note first that $\PP(\Omega_n  \cap \{\sup_{|t-s| < 2\Delta_n}| N_t - N_s| = 1)\rightarrow 1$ if $\lambda_n^2 \Delta_n \rightarrow 0$. We first work under Assumption \ref{assumpestimatorn} (I).  In this case, we have $\beta_n \Delta_n \rightarrow 0$ and we claim that
\begin{equation} \label{inf sgn}
\mathbb{P}\Big(\min_{1\leq q \leq N_1}\tfrac{\Delta^n_{i\left(n,q\right)} X}{\Delta X_{T_q}}  \leq 0  \big) \rightarrow 0
\end{equation}
from which Lemma \ref{lemma sgn} readily follows.
For all $q \geq 1$, we have
\begin{align*}\tfrac{\Delta^n_{i\left(n,q\right)} X}{\Delta X_{T_q}} &= 1 + \tfrac{\Delta^n_{i\left(n,q\right)} X^c}{\Delta X_{T_q}} - \tfrac{\left(1-e^{-\beta_n \Delta_n}\right) Z^{\beta}_{i\left(n,q\right)-1}}{\Delta X_{T_q}} \\
&\geq 1 - \tfrac{|\Delta^n_{i\left(n,q\right)} X^c|}{|\Delta X_{T_q}|} - \tfrac{\beta_n \Delta_n|Z^{\beta}_{i\left(n,q\right)-1}|}{|\Delta X_{T_q}|}\\
&\geq 1 - \frac{\max_{i \in {\mathcal A_n^c}} |\Delta^n_{i} X^c|}{\min_{1 \leq i \leq N_1} |\Delta X_{T_{i}}|} - \frac{\beta_n \Delta_n\max_{1 \leq i \leq n} |Z^{\beta}_{t_{i,n}}|}{\min_{1 \leq i \leq n} |\Delta X_{T_{i}}| }.
\end{align*}
Thus, 
$$
\mathbb{P}\Big(\min_{1\leq q \leq N_1}\tfrac{\Delta^n_{i\left(n,q\right)} X}{\Delta X_{T_q}}  \leq 0  \big) \leq \mathbb{P}\Big(\max_{i \in \mathcal A_n^c} |\Delta^n_{i} X^c|  + \beta \Delta_n \max_{1 \leq i \leq n} |Z^{\beta}_{t_{i,n}}|  \geq  \min_{1 \leq i \leq N_1}|\Delta X_{T_{i}}|  \Big)
$$
and this term converges to $0$ in the same way as the terms \eqref{dominf2} and \eqref{dominf3}. The convergence \eqref{inf sgn} follows and Lemma \ref{lemma sgn} is proved. The case where Assumption \ref{assumpestimatorn} (II) is fulfilled corresponds to showing the convergence \eqref{convsupb2c1} and the proof follows likewise.
\end{proof}

We are ready to prove Theorem  \ref{estimatorcj}.\\ 

\noindent {\it Step 1).} We prove the result on $\Omega_n\cap \big\{\mathrm{sgn}\left(\Delta X_{T_q}\right)  = \mathrm{sgn}(\Delta_{i\left(n,q\right)}^n X) \; \text{for every}\; 1\leq q \leq N_1\big\} \cap \{\sup_{|t-s| < 2\Delta_n}| N_t - N_s| = 1\big\}$ thanks to Lemma \ref{lemma sgn}. On this event, the quantity
\[\widehat s_n =  -\frac{\sum_{q=1}^{\widehat \lambda_n} \mathrm{sgn}(\Delta^n_{\mathcal{I}_n\left(q\right)} X) \big( \Delta_{\mathcal{I}_n\left(q\right)+1}^n X  + 2\Delta_n\sum_{j=1}^{q-1} \Delta^n_{\mathcal{I}_n\left(j\right)} X\big)}{\sum_{q=1}^{\widehat \lambda_n}\mathrm{sgn}(\Delta^n_{\mathcal{I}_n\left(q\right)} X) \Delta_{\mathcal{I}_n\left(q\right)}^n X}\]
is equal to 
 \[-\frac{\sum_{q=1}^{N_1} \mathrm{sgn}(\Delta^n_{i(n,q)} X) \big( \Delta_{i(n,q)+1}^n X  + 2\Delta_n\sum_{j=1}^{q-1} \Delta^n_{i\left(n, j \right)} X\big)}{\sum_{q=1}^{N_1}\mathrm{sgn}(\Delta^n_{i(n,q)} X) \Delta_{i(n,q)}^n X}\]
and
$$
\widehat s_n^{\mathrm{oracle}} = -\frac{\sum_{q=1}^{N_1} \mathrm{sgn}(\Delta X_{T_q}) \big( \Delta_{i\left(n,q\right)+1}^n X  + 2\Delta_n\sum_{j=1}^{q-1} \Delta X_{T_j}\big)}{\sum_{q=1}^{N_1} \mathrm{sgn}(\Delta X_{T_q})  \Delta_{i\left(n,q\right)}^n X}.
$$
It follows that $\widehat s_n = \widehat s_n^{\mathrm{oracle}} + \mathcal R_n^{(1)}$, with
$$\mathcal R_n^{(1)} =  2\Delta_n \frac{\sum_{q=1}^{N_1} \mathrm{sgn}(\Delta X_{T_q}) \sum_{j=1}^{q-1} (\Delta X_{T_j} - \Delta^n_{i\left(n,j\right)}X)}{\sum_{q=1}^{N_1} \mathrm{sgn}(\Delta X_{T_q})  \Delta_{i\left(n,q\right)}^n X} = I + II,
$$
with
$$I = 2\Delta_n\frac{\sum_{q=1}^{N_1} \mathrm{sgn}(\Delta X_{T_q}) \sum_{j=1}^{q-1} \Delta^n_{i\left(n,j\right)} \int_0^t (1-e^{-\beta_n(t-s)})x \underline{p}(dt,dx)}{\sum_{q=1}^{N_1} \mathrm{sgn}(\Delta X_{T_q})  \Delta_{i\left(n,q\right)}^n X}$$
and
$$II = - 2\Delta_n \frac{\sum_{q=1}^{N_1} \mathrm{sgn}(\Delta X_{T_q}) \sum_{j=1}^{q-1} \Delta^n_{i\left(n,j\right)} X^c}{\sum_{q=1}^{N_1} \mathrm{sgn}(\Delta X_{T_q})  \Delta_{i\left(n,q\right)}^n X}. $$
say. \\

\noindent {\it Step 2).} We quickly study each term separately. The term $I$ further splits into $I=III+IV$, with
\begin{align*}
III & = 2\Delta_n \frac{\sum_{q=1}^{N_1} \mathrm{sgn}(\Delta X_{T_q}) \sum_{j=1}^{q-1} \Delta X_{T_j}(1-e^{-\beta_n(t_{i\left(n,j\right),n} - T_j)})}{\sum_{q=1}^{N_1} \mathrm{sgn}(\Delta X_{T_q})  \Delta_{i\left(n,q\right)}^n X}
,\\
IV & = 2\Delta_n\frac{\sum_{q=1}^{N_1} \mathrm{sgn}(\Delta X_{T_q}) \sum_{j=1}^{q-1} Z_{t_{i\left(n,j\right)-1,n}}\left(1-e^{-\beta_n \Delta_n}\right)}{\sum_{q=1}^{N_1} \mathrm{sgn}(\Delta X_{T_q})  \Delta_{i\left(n,q\right)}^n X}.
\end{align*}
With the same kind of arguments as developed in the proof of Proposition \ref{estimatorc}, it is not difficult to see that $|IV| \lesssim \Delta_n^2 \lambda_n^2$ in probability. For the term $III$, we use the same kind of arguments and obtain
\begin{align*}
III & = \tilde{\mathcal{M}}_n+\beta_n\Delta_n\big(\sqrt{\lambda_n}\Delta_n+\sqrt{\Delta_n}\beta_n^{-1}\big)\mathcal R_n^{(2)},
\end{align*}
with 
$$\tilde{\mathcal{M}}_n = \frac{2\Delta_n \lambda_n \mathrm{sgn}(x) \star \nu \, (x \star \nu)\big(\tfrac{1}{2} - \tfrac{1-e^{-\beta_n\Delta_n}}{2\beta_n\Delta_n} 
 + \tfrac{1+e^{-\beta_n\Delta_n}}{2\beta_n} - \tfrac{1-e^{-\beta_n\Delta_n}}{\beta_n^2\Delta_n}\big)}{(|x| \star \nu)\tfrac{1-e^{-\beta_n \Delta_n}}{\beta_n \Delta_n}}$$
and $\mathcal R_n^{(2)}$ is bounded in probability. We also have $|II| \lesssim \Delta_n^{3/2}$ in probability. We omit the details.\\

\noindent {\it Step 3).} Finally we obtain the decomposition
$$\widehat s_n = \widehat s_n^{\mathrm{oracle}}+\tilde{\mathcal{M}}_n+\beta_n\Delta_n\big(\sqrt{\lambda_n}\Delta_n+\sqrt{\Delta_n}\beta_n^{-1}\big)\mathcal R_n^{(2)}$$
and we conclude by applying Proposition \ref{estimatorc}, replacing $\widehat s_n$ by $\widehat s_n^{\mathrm{oracle}}$ and studying the order of each term carefully.

\vip

\noindent {\bf Acknowledgements.} This research is supported by the department OSIRIS (Optimization, SImulation, RIsk and Statistics for Energy Markets) of EDF and by the FiME (Finance for Energy Markets) Research Initiative. The inputs of two anonymous referees are greatfully acknowledged.

\bibliographystyle{plain}
\bibliography{DeFeHo_Sub}
\end{document}